\documentclass[aop,MSNbibl,seceqn,dvips]{arximspdf}
\usepackage{graphicx}

\doi{10.1214/12-AOP768} %kopijuoti is PTS
\volume{41}
\issue{6}
\pubyear{2013}
\firstpage{4248}
\lastpage{4286}

\makeatletter

\newcommand{\varliminf}{\mathop{\underline{\lim}}}

\newcommand{\varlimsup}{\mathop{\overline{\lim}}}

\newtheorem{theorem}{Theorem}[section]
\newtheorem{lemma}[theorem]{Lemma}
\newtheorem{proposition}[theorem]{Proposition}
\newtheorem{corollary}[theorem]{Corollary}

\newproclaim{remark}[theorem]{Remark}

\renewcommand{\P}{\bP}
\newcommand{\rf}{J}
\newcommand{\dks}{d_{\xi}}
\newcommand{\hks}{h_{\xi}}
\newcommand{\pmu}{p}

\newcommand{\bE}{\mathbb{E}}
\newcommand{\bN}{\mathbb{N}}
\newcommand{\bP}{\mathbb{P}}
\newcommand{\bR}{\mathbb{R}}
\newcommand{\bZ}{\mathbb{Z}}

\newcommand{\xvec}{\mathbf{x}}
\newcommand{\yvec}{\mathbf{y}}
\newcommand{\vvec}{\mathbf{v}}
\newcommand{\vvbar}{\bar{\mathbf v}}
\newcommand{\uvec}{\mathbf{u}}
\newcommand{\om}{\omega}
\newcommand{\e}{\varepsilon}
\renewcommand{\l}{\ell}
\newcommand{\lmgf}{M}

\newcommand{\wt}{\widetilde}

\newcommand{\Vvv}{\operatorname{\mathbb{V}ar}}

\newcommand{\ddd}{\displaystyle}

\makeatother

\begin{document}
\begin{frontmatter}

\title{Large deviation rate functions for the partition function in a log-gamma
distributed random~potential}
\runtitle{Large deviations for log-gamma polymer}

\begin{aug}
\author[A]{\fnms{Nicos} \snm{Georgiou}\corref{}\ead[label=e1]{georgiou@math.utah.edu}\ead[label=u2,url]{http://www.math.utah.edu/\textasciitilde georgiou}}
\and
\author[B]{\fnms{Timo} \snm{Sepp\"al\"ainen}\thanksref{t1}\ead[label=e3]{seppalai@math.wisc.edu}\ead[label=u1,url]{http://www.math.wisc.edu/\textasciitilde seppalai}}
\runauthor{N. Georgiou and T. Sepp\"al\"ainen}
\affiliation{University of Utah and University of Wisconsin--Madison}
\address[A]{Mathematics Department\\
University of Utah\\
JWB\\
155 S 1400 E\\
Salt Lake City, Utah 84112-0090\\
USA\\
\printead{e1}\\
\printead{u2}}
\address[B]{Mathematics Department\\
University of Wisconsin--Madison\\
Van Vleck Hall\\
480 Lincoln Dr.\\
Madison, Wisconsin 53706-1388\\
USA\\
\printead{e3}\\
\printead{u1}} %adresu isvedimo komanda gale!
\end{aug}

\thankstext{t1}{Supported in part by NSF Grant DMS-10-03651 and by the
Wisconsin Alumni Research Foundation.}

% HISTORY:
\received{\smonth{11} \syear{2011}}
\revised{\smonth{4} \syear{2012}}

% ABSTRACT
%
\begin{abstract}
We study right tail large deviations of the logarithm of the
partition function for directed lattice paths in i.i.d. random
potentials. The main purpose is the derivation of explicit formulas for
the $1+1$-dimensional exactly solvable case with log-gamma distributed
random weights. Along the way we establish some regularity results for
this rate function for general distributions in arbitrary dimensions.
\end{abstract}

% KEYWORDS
%
\begin{keyword}[class=AMS]
\kwd[Primary ]{60K37}
\kwd[; secondary ]{60K35}
\kwd{60F10}
\end{keyword}
\begin{keyword}
\kwd{Directed polymer in random environment}
\kwd{partition function}
\kwd{large deviations}
\kwd{random walk in random potential}
\end{keyword}

\end{frontmatter}

%s1 #&#
\section{Introduction}\label{intro}

We study a version of the model called \textit{directed polymer in a
random environment}
where a fluctuating path is coupled with a random environment.
This model was introduced in the statistical physics literature in
\cite{huse-henl} and early
mathematically rigorous work followed
in \cite{bolt-cmp-89,imbr-spen}. We consider directed paths in the
nonnegative orthant
$\bZ_+^d$ of the $d$-dimensional integer lattice. The paths are
allowed nearest-neighbor
steps oriented along the coordinate axes. A random weight $\om(\uvec)$
is attached to each lattice point $\uvec\in\bZ_+^{d}$.
Together the weights form the \textit{environment} $\om= \{
\om(\uvec)\dvtx
\uvec\in\bZ_+^{d}\}$.
The space of environments is denoted by $\Omega$.
$\bP$ is a probability measure on $\Omega$ under which
the weights $ \{\om(\uvec)\}$ are i.i.d. random variables.

For $\vvec,\uvec\in\bZ^d_+$ such that $\vvec\le\uvec$
(coordinatewise ordering),
the set of admissible paths from $\vvec$ to $\uvec$ with $\vert
\uvec- \vvec\vert_1= m$ is
%
%e1.1 #&#
\begin{eqnarray}\label{Pi1}
\Pi_{\vvec,\uvec} &=& \bigl\{ x_\centerdot= \{ \vvec=x_0,
x_1,\ldots, x_{m} = \uvec\}\dvtx \forall k, x_k
\in\bZ^d_+ \mbox{ and }
\nonumber\\[-8pt]\\[-8pt]
&&\hspace*{93pt} x_{k+1} - x_k \in\{\mathbf e_i\dvtx 1\le i\le d\}
\bigr\},
\nonumber
\end{eqnarray}
where $e_i $ is the $i$th standard basis vector of $\bR^d$.
The \textit{point-to-point partition function} is
%
%e1.2 #&#
\begin{equation}\label{parf}
Z_{\vvec, \uvec} = \sum_{ x_\centerdot\in\Pi_{\vvec, \uvec} }e^{\sum_{j=1}^m \om
(x_j)}.
\end{equation}
This is the normalization factor in the \textit{quenched polymer distribution}
%
%e1.3 #&#
\begin{equation}\label{queme}
Q_{\vvec, \uvec}(x_\centerdot) = Z_{\vvec, \uvec}^{-1} \prod
_{j=1}^{m}e^{\om(x_j)},
\end{equation}
which is a probability distribution on the paths in the set $\Pi_{\vvec, \uvec}$.
When paths start at the origin ($\vvec= \mathbf0$), we drop $\vvec$
from the notation;
$Z_{\uvec}=Z_{\mathbf0,\uvec}$ and $\Pi_{\uvec}=\Pi_{\mathbf
0,\uvec}$.
Note that the weight
at the starting point $x_0$ was not included in the sum in the exponent
in (\ref{parf}).
This makes no difference for the results.
Sometimes it is convenient to include this weight, and then we write
$Z^\square_{\vvec, \uvec} = e^{\om(\vvec)} Z_{\vvec, \uvec}$
where the superscript
$\square$ reminds us that all weights in the rectangle are included.

In the polymer model one typically studies fluctuations of the path and
fluctuations
of $\log Z_{\uvec}$. This paper considers only $\log Z_{\uvec}$. Specifically
our main object of interest is the right tail large deviation rate function
%
%e1.4 #&#
\begin{equation}\label{introJ}
\rf_{\uvec}(r) = -\lim_{n\rightarrow\infty} n^{-1}\log\bP\{ \log
Z_{\lfloor{n\uvec}\rfloor} \ge nr\}
\end{equation}
for $\uvec\in\bR_+^d$, $r\in\bR$. Throughout we denote the
floor of a vector as $\lfloor{n\yvec}\rfloor= (\lfloor
{ny_1}\rfloor, \lfloor{ny_2}\rfloor,\ldots,
\lfloor{ny_d}\rfloor)$.
This function $J$ exists very generally for superadditivity reasons,
and in Section
\ref{secreg} we establish some of its regularity properties.

The focus of the
paper is an exactly solvable case where $d=2$ and
$-\om(\uvec)$ is log-gamma distributed. By ``exactly solvable'' we
mean that
special properties of the log-gamma case permit explicit computations, such
as a formula for the limiting point-to-point free energy
%
%e1.5 #&#
\begin{equation}\label{p1}
p(\yvec) = \lim_{n\to\infty} n^{-1} \log Z_{\lfloor{n\yvec
}\rfloor} ,\qquad\mbox{$\bP$-a.s.}
\end{equation}
and fluctuation exponents \cite{sepp-poly}. In the same spirit, in
this paper we compute explicit
formulas for the rate function $J$ and other related quantities in the
context of
the $1+1$-dimensional log-gamma polymer.

One can also consider point-to-line partition functions over all
directed paths of a fixed
length. For $m\in\bN$ the partition function is defined by
%
%e1.6 #&#
\begin{equation}\label{totparf}
Z^{\mathrm{line}}_m= \sum_{\uvec\in\bZ_+^d\dvtx \vert\uvec\vert_1=m}
Z_{\uvec}.
\end{equation}
Due to the $n^{-1}\log$ in front, in the results we look at
$Z^{\mathrm{line}}_m$
behaves like the maximal $Z_{\uvec}$ over $ \vert\uvec\vert_1=m$.

Some comments are in order.

There are currently three known exactly solvable directed polymer
models, all
in $1+1$ dimensions: the two with a discrete aspect are (i) the log-gamma
model introduced in \cite{sepp-poly}, and
(ii) a
model introduced in \cite{oconn-yor-01} where
the random environment is a collection
of Brownian motions. Some fluctuation exponents were derived for the
second model
in \cite{sepp-valk-10}, and it has been further studied in \cite{oconn-toda}
via a connection with the quantum Toda lattice. This Brownian model possesses
structures similar to those in the log-gamma model, so we expect that
the results
of the present paper could be reproduced for the Brownian model.

The third exactly solvable model is the continuum directed random
polymer \cite{Amir-Cor-Qual} that is expected to be a universal
scaling limit for a large class of polymer models; see \cite{Corwin}
for a recent review.

Usually the directed lattice polymer model is placed in a space--time
picture where the
paths are oriented in the time direction. (See articles and lectures
\mbox{\cite
{come-shig-yosh-03,come-shig-yosh-04,come-yosh-aop-06,denholl-polymer}} for
recent results and reviews of the general case.) In two dimensions (1
time${} + {}$1 space dimension),
the space--time picture is the same as our purely spatial picture, up to
a $45^\circ$ rotation
of the lattice and a change of lattice indices. The temporal aspect is not
really present in our work. So we have not separated a time dimension,
but simply
regard the paths as directed lattice paths.

Another standard feature of directed polymers that we have omitted is
the inverse temperature
parameter $\beta\in(0,\infty)$ that appears as a multiplicative
constant in front of the weights:
$ Z^{\beta}_{\vvec, \uvec} = \sum_{ x_\centerdot\in\Pi_{\vvec,
\uvec} }\exp\{ {\beta\sum_{j=1}^m \om(x_j)}\} $. For a fixed
weight distribution, $\beta$ modulates the strength
of the coupling between the walk and the environment. It is known that
in dimension
$1+3$ and higher, there can be a phase transition. By contrast, in low
dimensions ($1+1$ and $1+2$), the model is in the so-called strong
coupling regime for all
$0<\beta<\infty$ \cite{come-varg-06,laco-10}. The $\beta$ parameter
plays no role in the present work and has a fixed value $\beta=1$.
This is the unique $\beta$ value that turns the log-gamma model into
an exactly solvable model.

The techniques of the current paper are entirely probabilistic and rely
on the
stationary version of the log-gamma model. It can be expected that as a
combinatorial approach to this model, fully developed \cite
{corw-oconn-sepp-zygo},
more complete results
and alternative proofs for
the present results can be found.\vspace*{8pt}

\textit{Earlier literature.} Precise large deviation rate functions
for $\log Z$
in the case of directed polymers have not been derived in the past.
The strongest concentration inequalities can be found in recent references
\cite{Comets-Gregorio-arc,Liu-Watbled-2009,watbled-arc}.
The normalization of the left tail varies with the distribution of the weights
as demonstrated by \cite{Ben-Ari}, but the right tails have the same
normalization $n$.
Carmona and Hu \cite{Carmona-Hu-Gaussian}
have some bounds on the left tail of $\log Z$ in Gaussian
environments in dimensions $1+3$ and higher and for small enough $\beta$.
Similar bounds were proved later in \cite{moreno-2010} for bounded environments
using concentration inequalities for product measures.

For the exactly solvable zero-temperature models (i.e., last passage percolation
models), large deviation principles have been proved. For the longest increasing
path among planar Poisson points, an LDP for the length resulted from a
combination
of articles \cite{deus-zeit-99,kimj-96,loga-shep-77,sepp-ptrf-98}.
These results came before the advent of determinantal techniques.
For the corner growth model with geometric and exponential weights
\cite{joha} derived an LDP
in addition to the Tracy--Widom limit. An earlier right tail LDP
appeared in
\cite{sepp98ebp}.\vspace*{8pt}

\textit{Notation.} We collect some notation and conventions here for
easy reference. $\bN$ is for positive integers, $\bZ_{+}$ for
nonnegative integer, $\bR_{+}$ for nonnegative real numbers and $\bR^d_+$ is the set of all vectors with
nonnegative real coordinates.
Vector notation: elements of $\bR^d$ and $\bZ^d$ are $\vvec=
(v_1,v_2,\ldots, v_d)$. Coordinatewise ordering $\vvec\le\uvec$
means $v_1 \le u_1, v_2 \le u_2,\ldots, v_d \le u_d$. Particular
vectors are $\mathbf{1}= (1,1,\ldots, 1)$ and $\mathbf{0}=
(0,0,\ldots, 0)$. $\lfloor{\yvec}\rfloor= (\lfloor{y_1}\rfloor,
\lfloor{y_2}\rfloor,\ldots, \lfloor{y_d}\rfloor)$
where $\lfloor{y}\rfloor=\max\{n\in\bZ\dvtx n\le y\}$ is the integer
part of $y\in
\bR$.
The $\ell^1$ norm on $\bR^d$ is $\vert\vvec\vert_1=\vert v_1\vert
+\cdots
+\vert v_d\vert$.

The convex dual of a function $f\dvtx \bR\to(-\infty, \infty] $ is
$f^*(y)=\sup_{x\in\bR}\{ xy-f(x)\}$, and $f=f^{**}$ if and only if
$f$ is convex and lower
semicontinuous. We refer to \cite{rock-ca} for basic convex analysis.

The partition function $Z$ does not include the weight of the initial
point of the paths,
while $Z^\square$ does. In two dimensions we write $Z_{m,n}=Z_{(m,n)}$.

The usual gamma function is
$
\Gamma(\mu)=\int_{0}^{\infty} x^{\mu-1}e^{-x} \,dx
$
for $\mu>0$.
The digamma and trigamma functions are $\Psi_0=\Gamma'/\Gamma$
and $\Psi_1=\Psi_0'$. On $(0,\infty)$ $\Psi_0$ is increasing and concave
and $\Psi_1$ decreasing, positive and
convex, with $-\Psi_0(0+)=\Psi_1(0+)=\infty$.

%s2 #&#
\section{Large deviations for the log-gamma model}
\label{seclogg}

%s2.1 #&#
\subsection{The log-gamma model with i.i.d. weights}
\label{seciid-lg}
In this section we specialize to $d=2$ dimensions and the log-gamma
distributed weights.
Fix a positive real parameter $\mu$. This parameter remains fixed
through this entire
section, and hence is omitted from most notation.
In the log-gamma case we prefer to switch to multiplicative variables.
So the weight at point $(i,j)\in\bZ_+^2$ is $Y_{i,j}=e^{\om(i,j)}$
where the reciprocal $Y^{-1}$ has Gamma$(\mu)$ distribution.
Explicitly,
%
%e2.1 #&#
\begin{equation}\label{Gamu}
\bP\bigl\{Y^{-1} \ge s\bigr\}= \Gamma(\mu)^{-1} \int
_s^\infty x^{\mu-1}e^{-x} \,dx\qquad
\mbox{for $s\in\bR_+$.}
\end{equation}
As above, we write $Y$ for a generic random variable distributed as $Y_{i,j}$.
The digamma and trigamma functions give the
mean and variance, $\bE(\log Y)=-\Psi_0(\mu)$ and $\Vvv(\log
Y)=\Psi_1(\mu)$.

The logarithmic moment generating function (l.m.g.f.) of $\om=\log Y$ is
%
%e2.2 #&#
\begin{equation}\label{lmgfom}
\lmgf_{\mu}(\xi) = \log\bE \bigl( e^{\xi\log Y} \bigr)= \cases{\log
\Gamma(\mu-\xi) - \log\Gamma(\mu), &\quad $\xi\in (-\infty, \mu)$,
\cr
\infty, &\quad $\xi
\in[\mu,\infty)$.}\hspace*{-24pt}
\end{equation}

The point-to-point
partition function for directed paths from $(0,0)$ to $(m,n)$ is
%
%e2.3 #&#
\begin{equation} \label{parf2}
Z_{m,n} = \sum_{x_{\cdot}\in\Pi_{(m,n)}}\prod
_{j=1}^{m+n}Y_{x_j}.
\end{equation}
Note that we simplified notation by dropping the parentheses:
$Z_{m,n}=Z_{(m,n)}$.
For $(s,t)\in\bR_+^2$ the limiting free energy density exists by
superadditivity,
%
%e2.4 #&#
\begin{equation}\label{pmu1}
\pmu(s,t)= \lim_{n\to\infty} n^{-1}\log Z_{\lfloor
{ns}\rfloor, \lfloor{nt}\rfloor},\qquad \mbox{$
\bP$-a.s.}
\end{equation}
The limit is a finite constant. We begin by giving its exact value.
%
%th2.1 #&#
\begin{theorem}\label{llnTimo}
For $(s,t)\in\bR_+^2$ and $\mu\in(0,\infty)$, the limiting free
energy density
(\ref{pmu1}) is given by
%
%e2.5 #&#
\begin{equation}\label{pmu}
\pmu(s,t) = \inf_{0 < \rho< \mu}\bigl\{-s\Psi_0(\rho) - t
\Psi_0(\mu -\rho)\bigr\}.
\end{equation}
\end{theorem}
The value $\pmu(s,t)$ was already derived in \cite{sepp-poly} but the
proof was buried among
estimates for fluctuation exponents. In Section \ref{secpf1}
we sketch an elementary approach that utilizes special features of the
log-gamma model.
For the other explicitly solvable $1+1$-dimensional
polymer with Brownian environment, Moriarty and O'Connell \cite
{mori-oconn-07} computed the limiting free energy with a very different
large deviation approach.

%As an application of the large deviation results we compute below in
%limiting l.m.g.f. of $\log Z_{\fl{ns}, \fl{nt}}$.

The next result is a large deviation principle (LDP) for $\log
Z_{\lfloor{ns}\rfloor, \lfloor{nt}\rfloor}$
under normalization $n$.
The rate function is
%
%e2.6 #&#
\begin{equation}
\label{fdp} I_{s,t}(r) = \cases{ \ddd\sup_{\xi\in[0, \mu)} \Bigl\{ r\xi-
\inf_{\theta\in(\xi
,\mu)} \bigl( t\lmgf_{\theta}(\xi)-s\lmgf_{\mu-\theta}(-
\xi ) \bigr) \Bigr\}, &\quad $r \ge\pmu(s,t)$,
\vspace*{2pt}\cr
\infty, &\quad $r < \pmu(s,t)$.}\hspace*{-32pt}
\end{equation}
On the boundary ($s=0$ or $t=0$), the result reduces to i.i.d. large
deviations,
so we only consider $(s,t)$ in the interior of the quadrant.
%
%th2.2 #&#
\begin{theorem}\label{FULLldp}
Let $Y^{-1}\sim\operatorname{Gamma}(\mu)$ as in (\ref{Gamu}) and
$(s,t)\in(0,\infty)^2$. Then the distributions of $n^{-1}\log
Z_{\lfloor{ns}\rfloor, \lfloor{nt}\rfloor}$ satisfy a LDP
with normalization $n$ and rate function $I_{s,t}$. Explicitly, these
bounds hold
for any open set $G$ and any closed set $F$ in $\bR$:
%
%e2.7 #&#
\begin{equation}\label{fu2}
\varlimsup_{n\to\infty}n^{-1}\log\bP\bigl\{ n^{-1}\log
Z_{\lfloor
{ns}\rfloor, \lfloor{nt}\rfloor} \in F \bigr\} \le-\inf_{r\in F} I_{s,t}(r)
\end{equation}
and
%
%e2.8 #&#
\begin{equation}\label{fu1}
\varliminf_{n\to\infty}n^{-1}\log\bP\bigl\{ n^{-1}\log
Z_{\lfloor
{ns}\rfloor, \lfloor{nt}\rfloor} \in G \bigr\} \ge-\inf_{r\in G} I_{s,t}(r).
\end{equation}
On $[\pmu(s,t), \infty)$ the rate function $I_{s,t}$ is finite,
strictly increasing,
continuous and convex.
%
%f1 #&#
\begin{figure}

\includegraphics{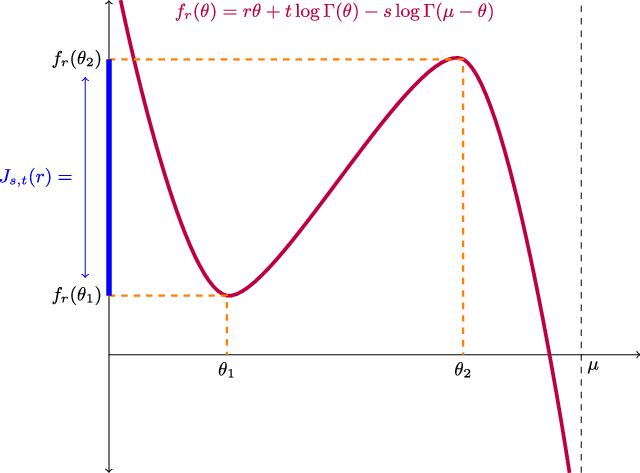}

\caption{Graphical representation of the solution to the variational problem
(\protect\ref{fdp}) that gives the rate function $\rf_{s,t}(r)
=f_r(\theta_2)-f_r(\theta_1) $. The curve $f_r(\theta)$ has the
same general shape as long as $r > p(s,t)$.}\label{bla}
\end{figure}
In particular, the unique zero of $I_{s,t}(r)$ is at $r=\pmu(s,t)$.
The right tail rate defined in (\ref{introJ}) is given by
%
%e2.9 #&#
\begin{equation}
\label{Jgam} \rf_{s,t}(r) =\cases{0, &\quad $r\in\bigl(-\infty, \pmu(s,t)\bigr]$,
\vspace*{2pt}\cr
I_{s,t}(r), &\quad $r\in\bigl[\pmu(s,t), \infty\bigr)$.}
\end{equation}
\end{theorem}
%
%re2.3 #&#
\begin{remark}
From a computational point of view, the solution to the variational
problem in (\ref{fdp}) can be computed
by
\[
I_{s,t}(r) = \sup_{0<\theta< \mu}\Bigl\{ f_r(\theta) -
\inf_{0<z \le
\theta} f_r(z) \Bigr\}= f_r(
\theta_2) - f_r(\theta_1),
\]
where
\[
f_r(\theta) = r\theta+ t\log\Gamma(\theta) - s\log\Gamma(\mu -
\theta),
\]
and for any $r> p(s,t)$,
$0< \theta_1 < \theta_2 < \mu$ are the solutions to the equation $
\frac{d}{d\theta} f_r(\theta) = 0$. (See Figure \ref{bla}.) This again
implies that the rate function is strictly positive as long as $r>p(s,t)$.
\end{remark}

%re2.4 #&#
\begin{remark} We do not address the precise large deviations in the left
tail, that is, in the range $r< \pmu(s,t)$. We expect the correct
normalization to be $n^2$. (Personal communication from I. Ben-Ari.)
Presently we do not have a technique for computing the rate function in that
regime. We include the trivial part $ I_{s,t}(r)=\infty$ for $r< \pmu
(s,t)$ in the theorem
so that we can compute the limiting l.m.g.f. by
a straightforward application of Varadhan's theorem.
\end{remark}

%fluctuation exponents vs. rate function expansion.}

Define for $\xi\in\bR$,
%
%e2.10 #&#
\begin{equation}\label{lmgf2}
\Lambda_{s,t}(\xi) = \lim_{n\to\infty} n^{-1} \log\bE
e^{ \xi\log Z_{\lfloor{ns}\rfloor, \lfloor
{nt}\rfloor}}.
\end{equation}

%co2.5 #&#
\begin{corollary}\label{lmgf-thm1} Let $\xi\in\bR$. Then the limit in (\ref{lmgf2})
exists and is given by
%
%e2.11 #&#
\begin{equation}\label{finpieces}
\Lambda_{s,t}(\xi)= I^*_{s,t}(\xi) = \cases{ \pmu(s,t)\xi, &\quad
$\xi< 0$,
\vspace*{2pt}\cr
\ddd\inf_{\theta\in(\xi,\mu)} \bigl\{ t\lmgf_{\theta}(\xi )-s
\lmgf_{\mu-\theta}(-\xi) \bigr\}, &\quad $0\le\xi<\mu$,
\vspace*{2pt}\cr
\infty, &\quad $\xi\ge\mu$.}\hspace*{-28pt}
\end{equation}
\end{corollary}
%
%re2.6 #&#
\begin{remark}\label{infrestr}
Symmetry of $\Lambda_{s,t}$ in $(s,t)$
is clear from (\ref{lmgf2}) but not immediately obvious in the $0\le
\xi<\mu$ case
of (\ref{finpieces}). It turns out that if $s\le t$ the infimum is
achieved at a unique
$\theta_0 \in[(\mu+ \xi)/2, \mu)$, and then for $\Lambda_{t,s}(\xi)$, the same infimum is uniquely
achieved at
$\theta_1 = \mu+ \xi- \theta_0\in(\xi, (\mu+ \xi)/2]$.
In the case $s=t$ a simple formula arises:
$ \Lambda_{t,t}(\xi)=2t(\log\Gamma(\frac{\mu-\xi}2)-\log\Gamma
(\frac{\mu+\xi}2))$.
\end{remark}
%
%re2.7 #&#
\begin{remark} The first case of (\ref{fdp}) gives $I_{s,t}$ as the dual
$\Lambda_{s,t}^*$, and the reader may wonder whether this is the logic
of the proof
of the LDP.
It is not, for we have no direct way to compute $\Lambda_{s,t}$. Instead,
Theorem \ref{FULLldp} is first proved in an indirect manner
via the stationary model described in the
next subsection, and then $\Lambda_{s,t}$ is derived by Varadhan's theorem.
\end{remark}

Let us also record the result for the point-to-line case.
It behaves like the point-to-point
case along the diagonal.
%
%co2.8 #&#
\begin{corollary}\label{free-rate}
Let $Y^{-1}\sim\operatorname{Gamma}(\mu)$ as in (\ref{Gamu}) and $s>0$.
Then the distributions of $\log Z^{\mathrm{line}}_{\lfloor
{ns}\rfloor}$
satisfy an LDP
with normalization $n$ and rate function $I_{s/2, s/2}$.
\end{corollary}
%
%re2.9 #&#
\begin{remark}
For $\e>0$ and $r= p(s,t) + \e$, one can show after some calculus
that there exists a nonzero constant $C = C_{s,t}(\mu)$ so that
\[
I_{s,t}(r) = C\e^{3/2} + o\bigl(\e^{3/2}\bigr).
\]
This suggests that $\Vvv(\log Z_{\lfloor{ns}\rfloor,\lfloor
{nt}\rfloor})$ is of order $n^{2/3}$. Rigorous upper bounds on the moments
$\bE|{\log Z_{\lfloor{ns}\rfloor,\lfloor{nt}\rfloor}-np(s,t)}|^p$
for $1\le p< 3/2$ can be
found in \cite{sepp-poly}, Theorem 2.4.

We computed the precise value of the constant $C$ for the point-to-line
rate function,
%
%e2.12 #&#
\begin{equation}\label{GMF-NG}
I_{1,1}(r) = \frac{4}{3}\frac{1}{\sqrt{|\Psi_2(\mu
/2)|}}\e^{3/2} + o
\bigl(\e^{3/2}\bigr),
\end{equation}
where $\Psi_2 = \Psi_0''$.
\end{remark}

%s2.2 #&#
\subsection{The stationary log-gamma model} \label{secstat-lg}
Next we consider the log-gamma model in a stationary situation that is special
to this choice of distribution.
Working with the stationary case is the key to explicit computations,
including all the previous
results, and provides some explanation for the formulas that arose for $I_{s,t}$
and $\Lambda_{s,t}$ in (\ref{fdp}) and (\ref{finpieces}).

The stationary model is created by
appropriately altering
the distributions of the weights on the boundaries of the quadrant $\bZ_+^2$.
We continue to use the parameter $\mu\in(0,\infty)$ fixed at the
beginning of this section,
and we introduce a second parameter $\theta\in(0,\mu)$.
Let the collection of independent weights $\{U_{i,0}, V_{0,j},
Y_{i,j}\dvtx
i,j \in\bN\}$
have the following marginal distributions:
%
%e2.13 #&#
\begin{eqnarray}\label{distr1}
U_{i,0}^{-1}&\sim&\operatorname{Gamma}(\theta),\qquad V_{0,j}^{-1}
\sim\operatorname{Gamma}(\mu-\theta)\quad \mbox{and}\nonumber\\[-8pt]\\[-8pt]
Y_{i,j}^{-1}&\sim&
\operatorname{Gamma}(\mu).\nonumber
\end{eqnarray}
Define the partition function $Z^{(\theta)}_{m,n}$ by (\ref{parf2})
with the following weights:
at the origin $Y_{0,0}=1$, on the $x$-axis
$Y_{i,0} = U_{i,0}$, on the $y$-axis $Y_{0,j} = V_{0,j}$, and in the bulk
the weights $\{Y_{i,j}\dvtx i,j \in\bN\}$ are i.i.d. $\operatorname{Gamma}(\mu)^{-1}$
as before.
Equivalently, we can decompose the
stationary partition function
$ Z^{(\theta)}_{m,n}$ according to the exit point of
the path from the boundary
%
%e2.14 #&#
\begin{equation} \label{deco-1}
Z^{(\theta)}_{m,n} %= \sum_{x_{\cdot}\in\Pi_{(0,0),(\fl{ns}, \fl{nt})}}\prod_{j=1}^{
= \sum
_{k=1}^{m} \Biggl(\prod
_{i=1}^{k}U_{i,0} \Biggr)
Z^{\square}_{(k,1),(m,n)} +\sum_{\l=1}^{n}
\Biggl(\prod_{j=1}^{\l}V_{0,j}
\Biggr) Z^{\square}_{(1,\l),(m,n)}.
\end{equation}

The symbols $U_{i,0}$ and $V_{0,j}$ were at first introduced for the
boundary weights
to highlight the change of distribution. Next let us define for all
$(i,j)\in\bZ_+^2\setminus\{(0,0)\}$,
%
%e2.15 #&#
\begin{equation}\label{fractionid}
U_{i,j}= \frac{Z^{(\theta)}_{i,j}}{Z^{(\theta)}_{i-1,j}}
\quad\mbox{and}\quad V_{i,j}=
\frac{Z^{(\theta)}_{i,j}}{Z^{(\theta)}_{i,j-1}}.
\end{equation}
Note that this property was already built into the boundaries because,
for example, $Z^{(\theta)}_{i,0}=U_{1,0}\cdots U_{i,0}$.
The key result that allows explicit calculations for this model is the
following.
%
%pr2.10 #&#
\begin{proposition}\label{burke-pr}
For each $(i,j)\in\bZ_+^2\setminus\{(0,0)\}$, we
have the following marginal distributions:
$U_{i,j}^{-1}\sim \operatorname{Gamma}(\theta)$ and $V_{i,j}^{-1}\sim
\operatorname{Gamma}(\mu
-\theta)$.
For any fixed $n\in\bZ_+$, the variables $\{U_{i,n}\dvtx i\in\bN\}$ are
i.i.d., and
for any fixed $m\in\bZ_+$, the variables $\{V_{m,j}\dvtx j\in\bN\}$ are i.i.d.
\end{proposition}
This is a special case of Theorem 3.3 in \cite{sepp-poly}, where the
independence
of these weights along more general down-right lattice paths is established.
Proposition~\ref{burke-pr} is the only result from \cite{sepp-poly}
that we use.
It follows in an elementary
fashion from the properties of the gamma distribution.

As an immediate application we can write
%
%e2.16 #&#
\begin{equation}\label{sum-1}
n^{-1}\log Z^{(\theta)}_{\lfloor{ns}\rfloor, \lfloor{nt}\rfloor} =n^{-1}\sum
_{j=1}^{\lfloor{nt}\rfloor} \log V_{0,j} +
n^{-1}\sum_{i=1}^{\lfloor{ns}\rfloor} \log
U_{i,\lfloor{nt}\rfloor}
\end{equation}
as a sum of two sums of i.i.d. variables, and from this compute
%
%e2.17 #&#
\begin{equation}\label{meanlog}
\bE\bigl(\log Z^{(\theta)}_{m,n}\bigr)=m\bE(\log U)+n\bE(\log V) =
-m\Psi_0(\theta) - n \Psi_0(\mu-\theta)\hspace*{-25pt}
\end{equation}
and obtain the law of large numbers,
%
%e2.18 #&#
\begin{equation}\label{th-lln}\qquad
n^{-1}\log Z^{(\theta)}_{\lfloor{ns}\rfloor, \lfloor{nt}\rfloor} \to p^{(\theta)}(s,t)=
-s\Psi_0(\theta)-t\Psi_0(\mu-\theta) ,\qquad\mbox{$\bP$-a.s.}
\end{equation}

Note that the two sums on the right-hand side of (\ref{sum-1}) are not
independent of each other.
In fact, they are so strongly negatively correlated that the variance
of their sum is
of order $n^{2/3}$ \cite{sepp-poly}.
Comparison of (\ref{pmu}) and (\ref{th-lln}) reveals a variational
principle at work:
$\pmu(s,t)$ is the minimal free energy of a stationary system with
bulk parameter $\mu$.

%characteristic shape? Comment on fluctuation exponents vs. rate
%function expansion.}

Instead of the right tail large deviation rate function, we
give the asymptotic l.m.g.f. in the next result. Define
%
%e2.19 #&#
\begin{equation}\label{L}
\Lambda_{\theta, (s,t)}(\xi) = \lim_{n\to\infty} n^{-1} \log\bE
e^{ \xi\log Z^{(\theta)}_{\lfloor{ns}\rfloor, \lfloor{nt}\rfloor}}.
\end{equation}

%th2.11 #&#
\begin{theorem}\label{La-th-thm}
Let $s,t\ge0 $ and $0 < \theta< \mu$. Then the limit in (\ref{L})
exists for $\xi\ge0$
and is given by
%
%e2.20 #&#
\begin{equation}\label{La-th}\qquad
\Lambda_{\theta, (s,t)}(\xi) = \cases{ \max \bigl\{ s\lmgf_{\theta}(\xi)-
t\lmgf_{\mu-\theta}(-\xi), t\lmgf_{\mu-\theta}(\xi)- s
\lmgf_{\theta}(-\xi) \bigr\},
\vspace*{2pt}\cr
\hspace*{38pt}0 \le\xi< \theta\wedge(\mu-\theta)
\vspace*{2pt}\cr
\infty, \qquad \xi\ge\theta\wedge(\mu-\theta).}
\end{equation}
%
%For $\theta\ge\mu/2$, symmetry gives $\rf^*_{\theta,(s,t)}(\xi) =
\end{theorem}
%
%re2.12 #&#
\begin{remark} Let the parameters $0 < \theta< \mu$ be given. The
\textit{characteristic direction} is the choice
%
%e2.21 #&#
\begin{equation}\label{char3}
(s,t)=c\bigl(\Psi_1(\mu-\theta), \Psi_1(\theta)\bigr)
\qquad\mbox{for a constant $c>0$.}
\end{equation}
With this choice
the variance of $\log Z^{(\theta)}_{\lfloor{ns}\rfloor, \lfloor
{nt}\rfloor}$ is of order
$n^{2/3}$, while
in other directions the fluctuations of $\log Z^{(\theta)}_{\lfloor
{ns}\rfloor,
\lfloor{nt}\rfloor}$ have
order\vspace*{1pt} of magnitude $n^{1/2}$ and they are asymptotically Gaussian \cite
{sepp-poly}.
By this token, we would expect the large deviations in the
characteristic situation
to be unusual, while in the off-characteristic directions we would
expect the more
typical large deviations of order $e^{-n}$ in both tails. In Lemma \ref
{left-lm}(b) we
give a bound on the left tail that indicates superexponential decay under
(\ref{char3}). This also implies that if (\ref{char3}) holds, then
formula (\ref{La-th})
can be complemented with the case $\Lambda_{\theta, (s,t)}(\xi) =
p^{(\theta)}(s,t)\xi$
for $\xi\le0$. Presently we do not have further information about
these large deviations.
\end{remark}
%
%re2.13 #&#
\begin{remark} If the two sums in (\ref{sum-1}) were independent we
would have
$\Lambda_{\theta, (s,t)}(\xi) =s\lmgf_{\theta}(\xi)+ t\lmgf_{\mu
-\theta}(\xi)$.
Obviously (\ref{La-th}) reflects the strong negative correlation of
these sums,
but currently we do not have a good explanation (besides the proof!)
for the formula that arises.
\end{remark}

The maximum in (\ref{La-th}) comes from the choice of the first step
of the path:
either horizontal or vertical. Corresponding to this choice, define partition
functions
%
%e2.22 #&#
\begin{equation} \label{Zhor}
Z^{(\theta), \mathrm{hor}}_{\lfloor{ns}\rfloor, \lfloor{nt}\rfloor
}=\sum_{k=1}^{\lfloor{ns}\rfloor}
\Biggl(\prod_{i=1}^{k}U_{i,0}
\Biggr)Z^{\square}_{(k,1),(\lfloor{ns}\rfloor
,\lfloor{nt}\rfloor)}
\end{equation}
and
%
%e2.23 #&#
\begin{equation} \label{Zver}
Z^{(\theta), \mathrm{ver}}_{\lfloor{ns}\rfloor, \lfloor{nt}\rfloor
}=\sum_{\l=1}^{\lfloor{nt}\rfloor}
\Biggl(\prod_{j=1}^{\l
}V_{0,j}
\Biggr)Z^{\square}_{(1,\l),(\lfloor{ns}\rfloor,\lfloor
{nt}\rfloor)},
\end{equation}
together with l.m.g.f.'s
%
%e2.24 #&#
\begin{equation}\label{Lhor}
\Lambda^{\mathrm{hor}}_{\theta, (s,t)}(\xi) = \lim_{n\to\infty}
n^{-1} \log\bE e^{ \xi\log Z^{(\theta), \mathrm{hor}}_{\lfloor
{ns}\rfloor,
\lfloor{nt}\rfloor}}
\end{equation}
and
%
%e2.25 #&#
\begin{equation} \label{Lver}
\Lambda^{\mathrm{ver}}_{\theta, (s,t)}(\xi) = \lim_{n\to\infty}
n^{-1} \log\bE e^{ \xi\log Z^{(\theta), \mathrm{ver}}_{\lfloor
{ns}\rfloor,
\lfloor{nt}\rfloor}}.
\end{equation}
Then $Z^{(\theta)}_{\lfloor{ns}\rfloor, \lfloor{nt}\rfloor
}=Z^{(\theta), \mathrm
{hor}}_{\lfloor{ns}\rfloor, \lfloor{nt}\rfloor}+Z^{(\theta),
\mathrm{ver}}_{\lfloor{ns}\rfloor, \lfloor{nt}\rfloor}$
leads to
%
%e2.26 #&#
\begin{equation}\label{La-max}
\Lambda_{\theta, (s,t)}(\xi) = \Lambda^{\mathrm
{hor}}_{\theta, (s,t)}(\xi) \vee
\Lambda^{\mathrm{ver}}_{\theta, (s,t)}(\xi),
\end{equation}
which is the starting point for the proof of (\ref{La-th}).

The horizontal and vertical partition functions are in some sense
between the
stationary one and the one from (\ref{parf2}) with i.i.d. weights.
It turns out that these intermediate partition functions behave either
like the stationary
one or like the i.i.d. one, with a sharp transition in between, and
this holds both
at the level of the limiting free energy density and the l.m.g.f.
Let us focus on the horizontal case, the vertical case being the same
after the
swap $s\leftrightarrow t$ and $\theta\leftrightarrow\mu-\theta$.

Qualitatively, with $t$ fixed, when
$s$ is large $Z^{(\theta), \mathrm{hor}}_{\lfloor{ns}\rfloor,
\lfloor{nt}\rfloor}$
behaves like $Z^{(\theta)}_{\lfloor{ns}\rfloor, \lfloor{nt}\rfloor}$,
and when $s$ is small $Z^{(\theta), \mathrm{hor}}_{\lfloor
{ns}\rfloor, \lfloor{nt}\rfloor}$ behaves like
$Z_{\lfloor{ns}\rfloor, \lfloor{nt}\rfloor}$ from (\ref{parf2}).
The conditions\vspace*{1pt} for the transitions are the following:
%
%e2.27 #&#
\begin{equation}\label{trans1}
s\Psi_1(\theta)\ge t\Psi_1(\mu-\theta)
\end{equation}
and
%
%e2.28 #&#
\begin{equation}\label{trans2}
s \bigl(\Psi_0(\theta)-\Psi_0(\theta-\xi) \bigr)\ge t
\bigl(\Psi_0(\mu-\theta+\xi)-\Psi_0(\mu-\theta) \bigr).
\end{equation}
By the concavity of $\Psi_0$ and the fact that $\Psi_1 = \Psi_0'$,
(\ref{trans1}) implies (\ref{trans2}) for all $\xi\ge0$.
Assuming the limit exists for the moment, define
%
%e2.29 #&#
\begin{equation}\label{pres-hor}
p^{{(\theta), \mathrm{hor}}}(s,t)=\lim_{n\to\infty} n^{-1} \log
Z^{(\theta), \mathrm{hor}}_{\lfloor{ns}\rfloor, \lfloor{nt}\rfloor}.
\end{equation}
In this next theorem the functions $\pmu(s,t)$ and $\Lambda_{s,t}(\xi
)$ are the
ones defined by (\ref{pmu}) and (\ref{finpieces}).
%
%th2.14 #&#
\begin{theorem}\label{hor-thm}
Let $s,t \ge0$, $0<\theta< \mu$ and $0 \le\xi< \theta$.\vspace*{8pt}

\textup{(a)} The limit in (\ref{pres-hor}) exists and is given by
%
%e2.30 #&#
\begin{equation}\label{pres-hor1}
p^{(\theta),\mathrm{hor}}(s,t)= \cases{ p^{(\theta)}(s,t), &\quad if (\ref{trans1}) holds,
\cr
\pmu(s,t), &\quad if (\ref{trans1}) fails.}
\end{equation}

\textup{(b)} The limit in (\ref{Lhor}) exists and is given by
%
%e2.31 #&#
\begin{equation}\label{pres-hor2}
\Lambda^{\mathrm{hor}}_{\theta, (s,t)}(\xi) = \cases{ s\lmgf_{\theta}(
\xi)- t\lmgf_{\mu-\theta}(-\xi), &\quad if (\ref{trans2}) holds,
\cr
\Lambda_{s,t}(\xi), &\quad if (\ref{trans2}) fails.}
\end{equation}
\end{theorem}
%
%re2.15 #&#
\begin{remark}\label{La-rem5}
We saw in (\ref{pmu}) that the limiting free energy $\pmu(s,t)$ of
the i.i.d. model is the minimal
free energy of the stationary models with the same bulk parameter $\mu$.
This link does \textit{not} extend to the l.m.g.f.'s: for $0<\xi<\mu$,
$\Lambda_{s,t}(\xi)< \Lambda_{\theta, (s,t)}(\xi)$
for all $\theta\in(0,\mu)$.
We observe this
at the end of the proof of Theorem~\ref{La-th-thm} in Section \ref{secpf2}.
\end{remark}

%s3 #&#
\section{The right tail rate function in the general case}
\label{secreg}

The
proofs of the results for the log-gamma model utilize regularity properties
of the rate function $\rf$ of (\ref{introJ}). These properties can be
proved in
some degree of generality, and we do so in this section. So now we consider
%
%e3.1 #&#
\begin{equation}\label{parf51}
Z_{\uvec} = \sum_{ x_\centerdot\in\Pi_{\uvec} }e^{\sum_{j=1}^{\vert u\vert_1}
\om(x_j)}
\end{equation}
as defined in the \hyperref[intro]{Introduction}, with $\uvec\in\bZ_+^d$, general
$d\ge2$,
and general i.i.d. weights $\{\om(\uvec)\}$.
%The right tail rate function is defined by
% \uvec\in\bR^d\setminus\{\mathbf0\}, r\in\bR.
%

We assume
%
%e3.2 #&#
\begin{equation}\label{Cramerass}
\exists \xi>0 \mbox{ such that } \bE \bigl(e^{\xi|\om(\uvec
)|} \bigr)< \infty.
\end{equation}
This guarantees the existence of a Cram\'er large deviation rate
function defined by
%
%e3.3 #&#
\begin{equation}\label{c1}
I(r) = -\lim_{\e\rightarrow0}\lim_{n\rightarrow\infty} n^{-1} \log\bP \Biggl
\{ n^{-1}\sum_{i=1}^{n}\om(
\uvec_i) \in(r-\e, r+\e ) \Biggr\}.
\end{equation}
(Above $\{\uvec_j\}$ are any distinct lattice points.)
We state first the existence theorem for the limiting point-to-point
free energy density.
We omit the proof because similar superadditive and approximation arguments
appear elsewhere in our paper, and refer to \cite{NicossThesis}. Let
us also
point out that assumption
(\ref{Cramerass}) is unnecessarily strong for this existence result,
but our objective
here is not to optimize on this point.
%
%th3.1 #&#
\begin{theorem}\label{substd} Assume (\ref{Cramerass}).
There exists an event $\Omega_0 \subseteq\Omega$ of full
$\P$-probability on which the convergence
%
%e3.4 #&#
\begin{equation}\label{limZ17}
p(\yvec) =\lim_{n\to\infty} n^{-1} \log Z_{\lfloor{n\yvec}\rfloor}
\end{equation}
happens simultaneously for all $\yvec\in\bR_+^d$.
Limit (\ref{limZ17}) holds also in $L^1(\bP)$.
As a function of $\yvec$, $p$ is concave and continuous on $\bR_+^d$.
\end{theorem}

Next the right-tail LDP. To avoid issues of vanishing probabilities and
infinite values of the rate, we make the following further assumption:
%
%e3.5 #&#
\begin{equation}\label{J-ass7}
\forall r <\infty\qquad \P\bigl\{\om(\mathbf0)>r\bigr\}>0.
\end{equation}

%th3.2 #&#
\begin{theorem}\label{J-thm1}
Assume (\ref{Cramerass}) and (\ref{J-ass7}).
Then for $\uvec\in\bR_+^d\setminus\{\mathbf0\}$ and $r \in\bR$,
the following
$\bR_+$-valued limit exists:
%
%e3.6 #&#
\begin{equation}\label{psidef}
\rf_{\uvec}(r) = -\lim_{n\rightarrow\infty} n^{-1}\log\bP\{ \log
Z_{\lfloor{n\uvec}\rfloor} \ge nr\}.
\end{equation}
As a function of $(\uvec,r)$,
$ \rf$ is convex and continuous on
$ (\bR^d_+\setminus\{\mathbf0\})\times\bR$.
$\rf_\uvec(r)=0$ if and only if $r \le p(\uvec)$.
\end{theorem}

Let us also remark that the weight $\om(\mathbf0)$ at the origin is
immaterial:
the limit is the same for $Z^\square$, so
for $\uvec\in\bR_+^d\setminus\{\mathbf0\}$ and $r \in\bR$,
%
%e3.7 #&#
\begin{equation}\label{J-15}
\rf_{\uvec}(r) = -\lim_{n\rightarrow\infty} n^{-1}\log\bP\bigl\{
\log Z^\square_{\lfloor{n\uvec}\rfloor} \ge nr\bigr\}.
\end{equation}
We observe this at the end of the proof of Theorem \ref{J-thm1}.

With a further
assumption on the Cram\'er rate function of the weight distribution
defined in (\ref{c1}),
we can extend the continuity of $ \rf_{\uvec}$ to $\uvec=0$:
%
%e3.8 #&#
\begin{equation}\label{J-ass8}
\alpha_\infty=\lim_{x\nearrow\infty} x^{-1}I(x)<\infty.
\end{equation}
Equation (\ref{J-ass7}) is equivalent to requiring that $ I(x)<\infty$ for all
large enough $x$, so
of course (\ref{J-ass8}) requires (\ref{J-ass7}). The constant
$\alpha_\infty$ is the
limiting slope of $I$ at $\infty$ which exists by convexity. When
assumption (\ref{J-ass8})
is in force we define
%
%e3.9 #&#
\begin{equation}\label{Jzero}
\rf_{\mathbf0}(r) = \cases{ 0, &\quad $r\le0$,
\cr
\alpha_\infty r, &\quad
$r\ge0$.}
\end{equation}

%th3.3 #&#
\begin{theorem}\label{J-thm2} Under assumptions (\ref{Cramerass}) and (\ref
{J-ass8}), and with $\rf_{\mathbf0}$
defined by~(\ref{Jzero}),
$\rf_\uvec(r)$ is finite and continuous on $\bR_+^d\times\bR$.
\end{theorem}
%
%re3.4 #&#
\begin{remark}\label{lg-rem1} Assumption (\ref{J-ass8}) is in particular valid for the
log-gamma model. For $Y^{-1}\sim\operatorname{Gamma}(\mu)$ the Cram\'er rate function
for $\om=\log Y$ is
%
%e3.10 #&#
\begin{equation}\label{logcram}\hspace*{23pt}
I_{\mu}(r) %&= -\lim_{n\to\infty} n^{-1} \log\bP\Bigg\{ \sum_{i=1}^{n}Y_i \ge r
=-r\Psi^{-1}_0(-r)-
\log\Gamma\bigl(\Psi^{-1}_0(-r)\bigr)+\mu r+ \log\Gamma (
\mu),\qquad r\in\bR.
\end{equation}
The limiting slope on the right is $\alpha_\infty=\mu$, while the
limiting slope on the left
would be $\lim_{r\to-\infty}I'(r)=-\infty$. In this case $\rf_{\mathbf0}(r)$ is also the
``rate function'' for the single weight at the origin
%
%e3.11 #&#
\begin{equation}\label{lg-tail1}
\rf_{\mathbf0}(r)= -\lim_{n\rightarrow\infty} n^{-1}\log\bP\{ \log Y
\ge nr\}.
\end{equation}
\end{remark}

The remainder of this section proves Theorems \ref{J-thm1} and
\ref{J-thm2}, and then
we prove two further lemmas for later use.
\begin{pf*}{Proof of Theorem \ref{J-thm1}}
For $m,n\in\bR_+$, let $\xvec_{m,n} \in\{0,1\}^{d}$ so that
$\lfloor(m+\allowbreak n)\uvec\rfloor= \lfloor{m\uvec}\rfloor+\lfloor
{n\uvec}\rfloor+ \xvec_{m,n}$.
By superadditivity, independence and shift invariance,
%
%e3.12 #&#
\begin{eqnarray}\label{temp18}
&&\bP\bigl\{ \log Z_{\lfloor{(m+n)\uvec}\rfloor} \ge{(m+n)r}\bigr\}
\nonumber\\[-8pt]\\[-8pt]
&&\qquad \ge\bP\{ \log Z_{\lfloor{m\uvec}\rfloor}\ge mr\} \bP\{ \log Z_{\lfloor{n\uvec}\rfloor} \ge nr
\}%\label{logsupad}
\bP\{ \log Z_{\xvec_{m,n}} \ge0\}. %\label{estimprob}
\nonumber
\end{eqnarray}
By assumption (\ref{J-ass7}) there is a uniform lower bound
\mbox{$\bP\{ \log Z_{\xvec_{m,n}} \ge0\}\ge\rho>0$}.
Thus $t(n)= \log\bP\{ \log Z_{\lfloor{n\uvec}\rfloor}\ge nr \} $
is superadditive with a small uniformly bounded correction.
Assumption (\ref{J-ass7}) implies
that
$t(n)>-\infty$ for all $n\ge n_0$.
Consequently by superadditivity the rate function
%
%e3.13 #&#
\begin{equation} \label{Fekete}
\rf_{\uvec}(r) = -\lim_{n\rightarrow\infty} n^{-1}\log\bP\{ \log
Z_{\lfloor{n\uvec}\rfloor} \ge nr\}
% &= -\lim_{n\rightarrow\infty} n^{-1}\log\bP\{ H_{\fl{ny}}(\fl{ nt})
%<\fl{nx}\} \nonumber\\
\end{equation}
exists for $\uvec= (u_1,\ldots,u_d) \in\bR_+^d$ and $r\in\bR$.
The limit in (\ref{Fekete}) holds also as $n\rightarrow\infty$
through real values, not just integers.

Similarly we get convexity of $ \rf$
in $(\uvec, r)$. Let $\lambda\in(0,1)$ and assume $(\uvec, r)
=\lambda(\uvec_1, r_1)+(1-\lambda)(\uvec_2, r_2)$. Then
\begin{eqnarray*}
&&
n^{-1}\log\bP\{ \log Z_{\lfloor{n\uvec}\rfloor} \ge nr \} \\
&&\qquad\ge\lambda(\lambda
n)^{-1} \log\bP\{ \log Z_{\lfloor{n\lambda
\uvec_1}\rfloor} \ge n\lambda r_1\}
\\
&&\qquad\quad{}+ (1-\lambda) \bigl((1-\lambda) n\bigr)^{-1} \log\bP\bigl\{ \log
Z_{\lfloor{n(1-\lambda) \uvec_2}\rfloor} \ge n(1-\lambda) r_2\bigr\}+o(1)
\end{eqnarray*}
and letting $n \rightarrow\infty$ gives
%
%e3.14 #&#
\begin{equation}
\rf_{\uvec}(r) \le\lambda\rf_{\uvec_1}(r_1) + (1-
\lambda) \rf_{\uvec_2}(r_2).
\end{equation}

Finiteness of $\rf$ follows from (\ref{J-ass7}), so now we know $\rf
$ to
be a finite, convex function on $(\bR_+^d\setminus\{\mathbf0\})
\times\bR$.
This implies that $\rf$ is continuous in the interior of
$(\bR_+^d\setminus\{\mathbf0\}) \times\bR$ and
upper semicontinuous on the whole set $(\bR_+^d\setminus\{\mathbf0\}
) \times\bR$
\cite{rock-ca}, Theorems 10.1 and 10.2.

The law of large numbers for the free energy implies $\rf_{\uvec }(r)=0
$ for $r< p(\uvec)$ and then by continuity for $r\le p(\uvec)$. With a
minor adaptation of \cite{Comets-Gregorio-arc}, Proposition~3.1(b), we
get a concentration inequality: given $\uvec$, for $\e>0$ there exists
a constant $c>0$ such that
%
%e3.15 #&#
\begin{equation}\qquad
\bP\bigl\{ \vert{\log Z_{\lfloor{n\uvec}\rfloor}} - \bE\log Z_{\lfloor
{n\uvec}\rfloor} \vert\ge n\e\bigr\}\le2
\exp\bigl( -c\e^2 n\bigr) \qquad\mbox{for all $n\in\bN$.}
\end{equation}
\label{321}
Since $n^{-1} \bE\log Z_{\lfloor{n\uvec}\rfloor} \to p(\uvec)$,
this implies that
$\rf_{\uvec}(r)>0 $ for $r>p(\uvec)$.

We do a coupling proof for lower semicontinuity. Let $(\uvec,r)\to
(\vvec,s)$ in $(\bR_+^d\setminus\{\mathbf0\})\times\bR$. If each
coordinate $v_i>0$, then we have continuity
$\rf_\uvec(r)\to\rf_\vvec(s)$ because convexity already gives
continuity in the interior. Thus we may assume that some coordinates of
$\vvec$ are zero. Since coordinates can be permuted without changing
$\rf$, let us assume that $\vvec= (v_1,v_2,\ldots, v_k,0,\ldots,0) $
for a fixed $1\le k<d$ where $v_1,\ldots,v_k>0$. If eventually $\uvec$
is also of the form $\uvec= (u_1,u_2,\ldots, u_k,0,\ldots,0) $ for the
same $k$, then we are done by convexity-implied continuity again, this
time in the interior of $(\bR_+^k\setminus\{\mathbf0\})\times\bR$.
% (A finite convex function is continuous in the relative interior of a
%convex set \cite{rock-ca}, Thm. 10.1,; alternately think of restricting
%the path model to $k$ dimensions.)

The remaining case is the one where $u_1,\ldots, u_k>0$ and
$(u_{k+1},\ldots, u_d)\to\mathbf0$. We develop a family of couplings that
eliminates these $d-k$ last coordinates one by one, starting with
$u_d$, and puts
us back in the interior case with continuity.
Denote a lower-dimensional
projection by $\uvec_{1,k} = (u_1,u_2,\ldots, u_k)$.

The set of paths $\Pi_{\lfloor{n\uvec}\rfloor}$ is decomposed
according to the locations
of the $\lfloor{nu_d}\rfloor$ unit jumps in the $\mathbf e_d$-direction.
The projections of these locations form a vector $\pi$ from the set
\begin{eqnarray*}
\Lambda_{\lfloor{n\uvec}\rfloor} &=& \bigl\{ \pi= \bigl\{ \xvec^{i} \bigr
\}_{
i=0}^{\lfloor{nu_d}\rfloor+1}\in\bigl(\bZ_+^{d-1}
\bigr)^{\lfloor
{nu_d}\rfloor+2}\dvtx
\\
&&\hspace*{5pt}\mathbf0 = \xvec^{0} \le\xvec^{1} \le\cdots\le
\xvec^{\lfloor
{nu_d}\rfloor+1}= \lfloor{n\uvec_{1,d-1}}\rfloor \bigr\}.
\end{eqnarray*}
The partition function then decomposes according to the following jump
locations:
%
%e3.16 #&#
\begin{equation}\label{partdeco}
Z_{\lfloor{n\uvec}\rfloor} = \sum_{\pi\in\Lambda_{\lfloor
{n\uvec}\rfloor} } Z_{(\mathbf0, 0), (\xvec^{1}, 0)}
\prod_{i=1}^{\lfloor{nu_d}\rfloor} Z^{\square}_{(\xvec^{i}, i ), (\xvec^{i+1}, i )}
\equiv\sum_{\pi
\in\Lambda_{\lfloor{n\uvec}\rfloor} } Z_\pi,
\end{equation}
where the last equality defines the $d-1$-dimensional partition
functions $Z_\pi$.

For a fixed $\pi$, define a new environment $\wt\om$ indexed by $\bZ_+^{d-1}$ with
this recipe:
\begin{longlist}[(iii)]
\item[(i)] For $0\le i\le\lfloor{nu_d}\rfloor$: for $\yvec\in\bZ_+^{d-1}$
such that $\xvec^{i}\le\yvec\le\xvec^{i+1}$
but $\yvec\ne\xvec^{i}$, set $\wt\om(\yvec)=\om(\yvec,i)$.
\item[(ii)] $\wt\om(\mathbf0)=\om(\mathbf0,0)$ and for $1\le
i\le\lfloor{nu_d}\rfloor$, $\wt\om(\lfloor{n\uvec_{1,d-1}}\rfloor+ i\mathbf
{e}_{d-1}) =
\om(\xvec^{i}, i )$.
\item[(iii)] Pick all other $\wt\om(\yvec)$ independently of
everything else.
\end{longlist}
Now, keeping $\pi$ fixed, we project the paths down to $\bZ_+^{d-1}$
and create a partition function (marked by a tilde)
in the new environment $\wt\om$:
%
%e3.17 #&#
\begin{eqnarray}\label{coupling}
\log Z_{\pi} &=& \log Z_{(\mathbf0, 0), (\xvec^{1}, 0)} + \sum
_{i=1}^{\lfloor{nu_d}\rfloor} \log Z^{\square}_{(\xvec^{i}, i ),
(\xvec^{i+1},
i )}
\nonumber
\\
&=& \sum_{i=0}^{\lfloor{nu_d}\rfloor} \log Z_{(\xvec^{i}, i ),
(\xvec^{i+1}, i
)}
+ \sum_{i=1}^{\lfloor{nu_d}\rfloor}\om\bigl(
\xvec^{i}, i \bigr)
\nonumber\\[-8pt]\\[-8pt]
&=& \sum_{i=0}^{\lfloor{nu_d}\rfloor} \log\wt
Z_{\xvec^{i}, \xvec
^{i+1}} + \sum_{i=1}^{\lfloor{nu_d}\rfloor} \wt\om
\bigl(\lfloor{n\uvec_{1,d-1}}\rfloor+ i\mathbf {e}_{d-1}\bigr)
\nonumber
\\
&\le&\log\widetilde Z_{\lfloor{n\uvec_{1,d-1}}\rfloor+ \lfloor
{nu_d\mathbf{e}_{d-1}}\rfloor}.
\nonumber
\end{eqnarray}

%To count the number of vectors in a set of type $\Lambda_{\fl{n\uvec}}$
Introduce the continuous
functions ($1\le i < d$)
%
%e3.18 #&#
\begin{equation}\label{patheses}
F_{i}(\uvec) = \sum_{j=1}^{i-1}
\bigl( (u_j + u_{i})\log(u_j +
u_{i}) - u_j \log u_j - u_{i}
\log u_{i} \bigr).
\end{equation}
Counting the number of ways to decompose
the length from $0$ to $\lfloor{nu_i}\rfloor$ into $\lfloor
{nu_d}\rfloor+1$
segments and Stirling's formula give
%
%e3.19 #&#
\begin{eqnarray}\label{pathcard}\quad
m_0 &=& \vert\Lambda_{\lfloor{n\uvec}\rfloor}\vert= \prod
_{1\le i
\le d-1}\pmatrix{ {\lfloor{nu_i}\rfloor+
\lfloor{nu_d}\rfloor}
\cr
{\lfloor{nu_d}\rfloor+1}}= \exp
\bigl\{nF_{d}(\uvec) + o(n)\bigr\}
\nonumber\\[-8pt]\\[-8pt]
&\le& \exp\bigl\{nF_{d}(\uvec) + n\eta\bigr\},
\nonumber
\end{eqnarray}
where the last inequality is valid for large $n$
and we introduced a small $\eta>0$ that we can send to zero after
limits in $n$ have
been taken.
By a union bound and the coupling (\ref{coupling}) separately for each
$\pi\in\Lambda_{\lfloor{n\uvec}\rfloor}$,
\begin{eqnarray*}
-\rf_{\uvec}(r) &\le& \varlimsup_{n\rightarrow\infty}n^{-1} \log\sum
_{\pi\in\Lambda_{\lfloor{n\uvec}\rfloor}} \bP \{ \log Z_{\pi
} \ge nr - \log
m_0 \}
\\
&\le& \lim_{n\rightarrow\infty} \biggl( \frac{\log m_0}n
\\
&&\hspace*{27.7pt}{} + n^{-1}\log \bP \bigl\{ \log\widetilde{Z}_{\lfloor{n\uvec_{1,d-1}}\rfloor+
\lfloor{nu_d\mathbf{e}_{d-1}}\rfloor} \ge nr -
nF_{d}(\uvec) -n\eta \bigr\} \biggr)
\\
&=& F_{d}(\uvec) - \rf_{\uvec_{1,d-1}+ u_d\mathbf
{e}_{d-1}} \bigl(r-F_{d}(
\uvec)-\eta \bigr).
\end{eqnarray*}
In the last step above a little correction as in (\ref{temp18}) replaces
$\lfloor{n\uvec_{1,d-1}}\rfloor+ \lfloor{nu_d\mathbf
{e}_{d-1}}\rfloor$ with
$\lfloor{n\uvec_{1,d-1} + nu_d\mathbf{e}_{d-1}}\rfloor$.

Let $\widetilde{\uvec}_{1, d}=\uvec$ and for $1\le i<d$,
\[
\widetilde{\uvec}_{1, i}= \uvec_{1, i}+\sum
_{j=i+1}^{ d} u_j\mathbf
{e}_{i} \in\bZ_+^{i}.
\]
Proceeding inductively, we get the lower bound
%
%e3.20 #&#
\begin{equation}\label{lb9}
\rf_{\uvec}(r) \ge\rf_{\widetilde{\uvec}_{1,k}} \biggl(r-\sum
_{k+1\le i \le d} \bigl(F_{i}(\uvec)-\eta \bigr) \biggr) -\sum
_{k+1\le i
\le d}F_{i}(\uvec).
\end{equation}
On the right-hand side we have a rate function $\rf_{\widetilde{\uvec
}_{1,k}}$
with $\widetilde{\uvec}_{1,k}\to\vvec_{1,k}$ in the
interior of $\bR_+^k$. Thus we have continuity. We can first let $\eta
\searrow0$.
Then let $(\uvec,r)\to(\vvec,s)$. Note that $u_i\to0$ implies
$F_i(\uvec)\to0$.
Together all this gives the lower semicontinuity
\[
\varliminf_{(\uvec,r)\to(\vvec,s)} \rf_{\uvec}(r) \ge\rf_{\widetilde{\vvec}_{1,k}}(s) =
\rf_{\vvec
}(s).
\]

Now we know $\rf$ is continuous on all of $(\bR_+^d\setminus\{
\mathbf0\}) \times\bR$.

Let us observe limit (\ref{J-15}). From one side we have
\[
\bP\bigl\{ \log Z^\square_{\lfloor{n\uvec}\rfloor} \ge nr\bigr\} \ge\bP\{ \log
Z_{\lfloor{n\uvec}\rfloor} \ge nr\}\bP\bigl\{\om (\mathbf0) \ge0\bigr\}.
\]
From the other, pick a coordinate $u_i>0$, and for each $n$ an integer
$n<m_n<n+o(n)$ such that $2\mathbf e_i+\lfloor{n\uvec}\rfloor\le
\lfloor{m_n\uvec}\rfloor$.
For each $n$ fix a directed path $\{x^n_j\}$ from
$2\mathbf e_i+\lfloor{n\uvec}\rfloor$ to $\lfloor{m_n\uvec}\rfloor$.
Inequality
\[
\om(\mathbf e_i)+ \log Z^\square_{2\mathbf e_i, 2\mathbf
e_i+\lfloor{n\uvec}\rfloor} +\sum
_j \om\bigl(x^n_j\bigr)
\le\log Z_{\lfloor{m_n\uvec}\rfloor}
\]
gives
\[
\bP\bigl\{ \log Z^\square_{\lfloor{n\uvec}\rfloor} \ge nr\bigr\} \bP \biggl\{
\om(\mathbf e_i)+ \sum_j \om
\bigl(x^n_j\bigr) \ge0 \biggr\} \le\bP\{ \log
Z_{\lfloor{m_n\uvec}\rfloor} \ge nr\}.
\]
Assumption (\ref{J-ass7}) and the continuity of $\rf$ give the conclusion.
\end{pf*}\eject
\begin{pf*}{Proof of Theorem \ref{J-thm2}}
It remains to prove continuity at $(\mathbf0, s)$.
Let $(\uvec,r)\to(\mathbf0,s)$. Define the right-tail Cram\'er rate function
for $a>0$, $x\in\bR$:
\begin{eqnarray*}
\kappa_a(x) &=& -\lim_{n\rightarrow\infty} n^{-1} \log\bP \Biggl
\{ n^{-1}\sum_{i=1}^{\lfloor{na}\rfloor}\om
(x_i) \ge nx \Biggr\} \\
&=&\cases{ aI(x/a), &\quad $x\ge a\bE\bigl[\om(
\mathbf0)\bigr]$,
\vspace*{2pt}\cr
0, &\quad $x\le a\bE\bigl[\om(\mathbf0)\bigr]$.}
\end{eqnarray*}
Check that as $(a,x)\to(0,s)$, $\kappa_a(x)\to\rf_{\mathbf0}(s)$
defined by (\ref{Jzero}).

For upper semicontinuity, bound $Z_{\lfloor{n\uvec}\rfloor}$ below
by a single path
\[
\rf_\uvec(r)\le-\lim_{n\rightarrow\infty} n^{-1} \log\bP \Biggl\{
n^{-1}\sum_{i=1}^{\vert\lfloor{n\uvec
}\rfloor\vert_1}\om
(x_i) \ge nr \Biggr\} =\kappa_{\vert\uvec\vert_1}(r).
\]

For lower semicontinuity, permute the coordinates so that $u_1>0$ as
$\uvec\to\mathbf0$.
Apply (\ref{lb9}) after $\eta$ has been taken to zero:
\[
\rf_{\uvec}(r) \ge\rf_{u_1\mathbf{e}_{1}} \biggl(r-\sum
_{2\le i \le
d} F_{i}(\uvec) \biggr) -\sum
_{2\le i \le d}F_{i}(\uvec).
\]
Since $ \rf_{u_1\mathbf{e}_{1}}=\kappa_{u_1}$ we get the lower
semicontinuity.
\end{pf*}

Finally two lemmas for later use. The next one allows more general lattice
sequences for the right-tail LDP.
%
%le3.5 #&#
\begin{lemma}\label{lm-J10} Let $\yvec\in(0,\infty)^d$ and $\uvec_n\in\bZ_+^d$
be a sequence such that
$n^{-1}\uvec_n\to\yvec$. Then for $r\in\bR$,
%
%e3.21 #&#
\begin{equation}\label{J10}
\lim_{n\rightarrow\infty} n^{-1}\log\bP\{ \log Z_{\uvec_n} \ge nr\} =
- \rf_{\yvec}(r).
\end{equation}
\end{lemma}
\begin{pf} Let us use assumption (\ref{J-ass7}) again.
Since the coordinates of $\uvec_n$ and $\lfloor{n\yvec}\rfloor$ are
increasing
to $\infty$,
for each $n$ we can find $\ell_n$ and $m_n$ such that
$\lfloor{\ell_n\yvec}\rfloor\le\uvec_n\le\lfloor{m_n\yvec
}\rfloor$ and in such a way
that $ n-\ell_n$, $n-m_n$ are eventually $o(n)$. For each $n$ fix
directed paths
$\{x_{n,i}\}_{0\le i\le K_n} $ from $\lfloor{\ell_n\yvec}\rfloor$
to $\uvec_n$
and $\{x'_{n,j}\}_{0\le j\le K'_n}$ from $\uvec_n$ to $\lfloor
{m_n\yvec}\rfloor$. Then
\[
Z_{\lfloor{\ell_n\yvec}\rfloor} \cdot W_n \le Z_{\uvec_n} \le
Z_{\lfloor{m_n\yvec}\rfloor} \cdot\bigl(W'_n\bigr)^{-1},
\]
where
\[
\log W_n = \sum_{1\le i\le K_n}
\om(x_i) \quad\mbox{and}\quad \log W'_n = \sum
_{1\le i\le K'_n}\om\bigl(x'_i\bigr).
\]
Assumption $n^{-1}\uvec_n\to\yvec$ implies that $K_n$ and $K'_n$ are
also $o(n)$.

The estimates we need follow. For example,
\[
\bP\{ \log Z_{\lfloor{m_n\yvec}\rfloor} \ge nr\} \ge\bP\bigl\{\log W'_n
\ge0\bigr\} \bP \{ \log Z_{\uvec_n} \ge nr\}
\]
and then by assumption (\ref{J-ass7}) and the continuity of the rate function,
\[
\varlimsup_{n\to\infty} n^{-1}\log\bP\{ \log Z_{\uvec_n} \ge
nr\} \le\lim_{n\rightarrow\infty} n^{-1}\log\bP\{ \log Z_{\lfloor
{m_n\yvec}\rfloor} \ge
nr\} = - \rf_{\yvec}(r).
\]
Similarly for the complementary lower bound on $\varliminf$.
\end{pf}
%
%le3.6 #&#
\begin{lemma}\label{sumsind}
Suppose that for each $n$, $L_n$ and $Z_n$ are independent random
variables. Assume that the limits
%
%e3.22 #&#
%e3.23 #&#
\begin{eqnarray}
\label{cradev}
 {\lambda}(s) &=& -\lim_{n\rightarrow\infty}n^{-1}\log\bP\{ L_n
\ge ns \},
\\
\label{rightdev}
\phi(s) &=& -\lim_{n\rightarrow\infty}n^{-1}\log\bP
\{ Z_n \ge ns \}
\end{eqnarray}
exist and are finite for all $s\in\bR$. Assume that $\lambda
(a_\lambda)=\phi(a_\phi)=0$ for some
$a_\lambda$, $a_{\phi}\in\bR$. Assume also that $\lambda$ is continuous.
Then for $r\in\bR$
%
%e3.24 #&#
\begin{eqnarray}
\label{sumsind0}
&&
\lim_{n\rightarrow\infty} \frac{\log\bP\{ L_n + Z_n \ge nr \}}{n} \nonumber\\[-8pt]\\[-8pt]
&&\qquad=\cases{ \displaystyle -
\inf_{a_{\lambda}\le s \le r-a_{\phi}}\bigl\{ \phi(r-s)+ {\lambda}(s) \bigr\},
&\quad $r >
a_{\phi}+a_{\lambda}$,
\vspace*{1pt}\cr
0, &\quad $r \le a_{\phi}+a_{\lambda}$.}\nonumber
\end{eqnarray}
%
%Furthermore, if $\lambda$ is continuous on $\bR$ then the infimum in
\end{lemma}
\begin{pf}
The lower bound $\ge$ follows from
\[
\bP\{ L_n + Z_n \ge nr \} \ge\bP\{ L_n \ge
ns \} \bP\bigl\{ Z_n \ge n(r-s) \bigr\}.
\]
Since an upper bound $0$ is obvious, it remains to show the upper bound
for the case $r > a_{\phi}+a_{\lambda}$.
Take a finite partition $ a_{\lambda}= q_0 < \cdots< q_m = r -a_{\phi
} $.
Then use a union bound and independence:
\begin{eqnarray*}%\label{dec}
&&\bP\{ L_n + Z_n \ge nr \}
\\
&&\qquad\le\bP\{ L_n + Z_n \ge nr, L_n < n
q_0 \}
\\
&&\qquad\quad{} +\sum_{i=0}^{m-1}\bP\{ L_n +
Z_n \ge nr, nq_i \le L_n \le
nq_{i+1}\} +\bP\{ L_n \ge nq_m\}
\\
&&\qquad\le\bP\bigl\{ Z_n \ge n (r-q_0) \bigr\} + \sum
_{i=0}^{m-1}\bP\bigl\{ Z_n \ge
n(r -q_{i+1})\bigr\} \bP\{ L_n \ge nq_{i}\}
\\
&&\qquad\quad{} +\bP\{ L_n \ge nq_m\}.
\end{eqnarray*}
From this,
\begin{eqnarray*}%\label{oops}
&&\varlimsup_{ n \rightarrow\infty} n^{-1} \log\bP\{ L_n +
Z_n \ge nr \}
\\
&&\qquad \le- \min \Bigl\{ \phi(r-q_0), \min_{0\le i \le m-1}\bigl[\phi
(r-q_{i+1})+ {\lambda}(q_{i})\bigr], \lambda(q_m)
\Bigr\}. %
\end{eqnarray*}
Note that $\lambda(q_0)=\phi(r-q_m)=0$, refine the partition and use the
continuity of $\lambda$.
\end{pf}

%s4 #&#
\section{Proofs for the i.i.d. log-gamma model}
\label{secpf1}

In this section we prove the results of Section \ref{seciid-lg}.
Throughout this section the dimension $d=2$ and the weights satisfy
$Y_{i,j}^{-1}\sim\operatorname{Gamma}(\mu)$ as in (\ref{Gamu}). As before,
for $(s,t)\in\bR_+^2\setminus\{(0,0)\}$ define
the function $\rf_{s,t}$ by the limit
%
%e4.1 #&#
\begin{equation}\label{J2}
\rf_{s,t}(r) = -\lim_{n\rightarrow\infty} n^{-1}\log\bP\{ \log
Z_{\lfloor{ns}\rfloor, \lfloor{nt}\rfloor} \ge nr\},\qquad r\in\bR.
\end{equation}
At the origin set
%
%e4.2 #&#
\begin{equation}\label{Jzero3}
\rf_{0,0}(r)=\cases{ 0, &\quad $r\le0$,
\cr
\mu r, &\quad $r\ge0$.}
\end{equation}
Then, as observed in Remark \ref{lg-rem1}, the function $ \rf_{s,t}(r)$ is finite
and continuous at all $(s,t,r)\in\bR_+^2\times\bR$.

We begin with a lemma that proves
Theorem \ref{llnTimo}.
% and also confirms that $\rf_{s,t}$ is positive where it should be.
%
%le4.1 #&#
\begin{lemma}\label{Jzero-lm} For $(s,t)\in\bR^2_+$
the limiting free energy of (\ref{pmu}) satisfies
%
%e4.3 #&#
\begin{equation}\label{pmu4}
\pmu(s,t) = \inf_{0 < \theta< \mu}\bigl\{-s\Psi_0(\theta) - t
\Psi_0(\mu-\theta)\bigr\}.
\end{equation}
The infimum
is achieved at some $\theta$ because $\Psi_0(0+)=-\infty$.
\end{lemma}

\begin{pf} The proof anticipates some themes of the later LDP proof,
but in a simpler context. We already recorded the law of large numbers
(\ref{th-lln}).
%
%
%f2 #&#
\begin{figure}

\includegraphics{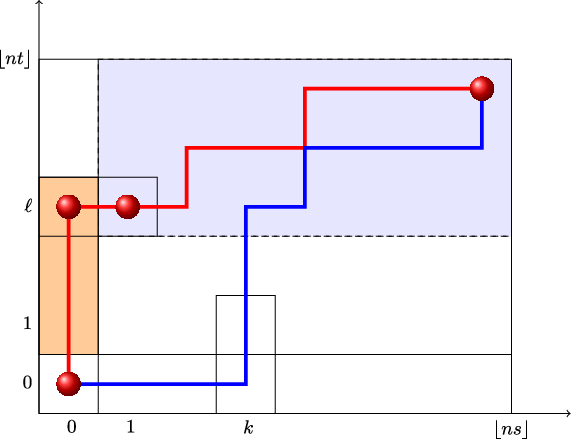}

\caption{Graphical representation of the decomposition in equation
(\protect\ref{deco}).}\label{figdec}
\end{figure}
The decomposition (see Figure \ref{figdec})
%
%e4.4 #&#
\begin{eqnarray}\label{deco}
Z^{(\theta)}_{\lfloor{ns}\rfloor, \lfloor{nt}\rfloor} %= \sum_{x_{\cdot}\in\Pi_{(0,0),(\fl{ns}, \fl{nt})}}\prod_{j=1}^{
&=& \sum
_{k=1}^{\lfloor{ns}\rfloor} \Biggl(\prod
_{i=1}^{k}U_{i,0} \Biggr)Z^{\square}_{(k,1),(\lfloor{ns}\rfloor,\lfloor{nt}\rfloor)}
\nonumber\\[-8pt]\\[-8pt]
&&{}+ \sum_{\l=1}^{\lfloor{nt}\rfloor} \Biggl(\prod
_{j=1}^{\l
}V_{0,j}
\Biggr)Z^{\square}_{(1,\l),(\lfloor{ns}\rfloor,\lfloor{nt}\rfloor)}\hspace*{-28pt}\nonumber
\end{eqnarray}
from (\ref{deco-1}) gives asymptotically
\begin{eqnarray*}
&&
\lim_{n\to\infty} n^{-1}\log Z^{(\theta)}_{\lfloor{ns}\rfloor,
\lfloor{nt}\rfloor}
\\
&&\qquad= \lim_{n\to\infty} \Biggl\{ \max_{1\le k\le\lfloor{ns}\rfloor
} \Biggl( n^{-1}
\sum_{i=1}^{k}\log U_{i,0} +
n^{-1}\log Z^{\square}_{(k,1),(\lfloor{ns}\rfloor,\lfloor
{nt}\rfloor)} \Biggr)
\\
&&\hspace*{28pt}\qquad\quad{} \vee\max_{1\le\ell\le\lfloor{nt}\rfloor} \Biggl( n^{-1} \sum
_{j=1}^{\ell}\log V_{0,j} + n^{-1}
\log Z^{\square}_{(1,\ell),(\lfloor{ns}\rfloor,\lfloor
{nt}\rfloor)} \Biggr) \Biggr\}.
\end{eqnarray*}
This can be coarse-grained with readily controllable errors of sums of
independent variables.
We omit the details since similar arguments appear elsewhere in the paper.
The conclusion is the alternative formula
%
%e4.5 #&#
\begin{eqnarray}
\label{th-lln2}
&&\lim_{n\to\infty} n^{-1}\log Z^{(\theta)}_{\lfloor{ns}\rfloor,
\lfloor{nt}\rfloor}
\nonumber\\
&&\qquad = \sup_{0\le a\le s}\bigl\{ -a\Psi_0(\theta) +\pmu(s-a,t) \bigr
\} \\
&&\qquad\quad{}\vee \sup_{0\le b\le t}\bigl\{ -b\Psi_0(\mu-\theta)
+\pmu(s,t-b) \bigr\}.
\nonumber
\end{eqnarray}

Take $s=t$, combine (\ref{th-lln}) and (\ref{th-lln2}), and use the symmetry
$\pmu(s,t)=\pmu(t,s)$ to get
\[
-t \bigl(\Psi_0(\theta)+\Psi_0(\mu-\theta) \bigr) =
\sup_{0\le a\le t} \bigl\{ -a \bigl(\Psi_0(\theta)\wedge
\Psi_0(\mu-\theta) \bigr) +\pmu(t-a,t) \bigr\}.
\]
Take $\theta\in(0,\mu/2]$ so that $\Psi_0(\theta)\le\Psi_0(\mu
-\theta)$
($\Psi_0$ is strictly increasing) and set $a=t-s$:
\[
-t \Psi_0(\mu-\theta) = \sup_{0\le s\le t} \bigl\{ s
\Psi_0(\theta) +\pmu(s,t) \bigr\}.
\]
Turn this into a convex duality through the change of variable
$v=\Psi_0(\theta)$:
%
%e4.6 #&#
\begin{equation}\label{dual33}\qquad
-t\Psi_0\bigl(\mu-\Psi_0^{-1}(v)\bigr)=
\sup_{0\le s\le t} \bigl\{ sv + \pmu (s,t) \bigr\},\qquad v \in\bigl(-\infty,
\Psi_0(\mu/2)\bigr].
\end{equation}

It follows from the limit definition of $\pmu(s,t)$ that it is concave
and continuous
in $s\in[0,t]$. Extend
$f(s)=-\pmu(s,t)$ to a lower semicontinuous convex function of $s\in
\bR$ by setting $f(s)=\infty$
for $s\notin[0,t]$. Then (\ref{dual33}) tells us that
\[
f^*(v)=-t\Psi_0\bigl(\mu-\Psi_0^{-1}(v)
\bigr) \qquad\mbox{for } v \in \bigl(-\infty, \Psi_0(\mu/2)\bigr].
\]
We can differentiate to get $\lim_{v\searrow-\infty} (f^*)'(v)=0$ and
$(f^*)'(\Psi_0(\mu/2))=t$. These derivative values imply that for
$s\in[0,t]$, the
supremum in the double convex duality can be restricted as follows:
\[
f(s)=\sup_{v \in(-\infty, \Psi_0(\mu/2)]} \bigl\{ vs - f^*(v)\bigr\}.
\]
Undoing the change of variables turns this equation into (\ref{pmu4})
which is thereby proved.
\end{pf}

The next lemma gives left tail bounds strong enough to imply
$ I_{s,t}(r)=\infty$ for $r< \pmu(s,t)$, and the same result for the
stationary
model. The proof is a straightforward coarse-graining argument.
We do not expect the results to be optimal.
%
%le4.2 #&#
\begin{lemma}\label{left-lm}
Fix $0< a < 1$. Then there exist constants $0<c,C<\infty$ that depend
on the parameters
given below, so that the following estimates hold:
\begin{longlist}[(a)]
\item[(a)] For $(s,t) \in(0,\infty)^2$ and $r < p(s,t)$,
%
%e4.7 #&#
\begin{equation}\label{left-5}
\bP\{\log Z_{\lfloor{ns}\rfloor, \lfloor{nt}\rfloor} \le nr \}
\le Ce^{-cn^{1+a}} \qquad\mbox{for all $n
\ge1$}.
\end{equation}
\item[(b)] For $(s,t)=\alpha(\Psi_1(\mu-\theta), \Psi_1(\theta))$ for some $\alpha>0$,
parallel to the characteristic direction, and $r< p^{(\theta)}(s,t)$,
%
%e4.8 #&#
\begin{equation}\label{left-stat}
\bP\bigl\{\log Z^{(\theta)}_{\lfloor{ns}\rfloor, \lfloor{nt}\rfloor} \le nr \bigr\} \le
Ce^{-cn^{1+a}} \qquad\mbox{for all $n\ge1$}.
\end{equation}
\end{longlist}
\end{lemma}
\begin{pf}
We give a proof of (b) with some details left sketchy. Part (a) has a
similar proof.
We bound $Z^{(\theta)}_{\lfloor{ns}\rfloor, \lfloor{nt}\rfloor}$
from below\vspace*{1pt} by considering
a subset of lattice paths,
arranged in a collection of i.i.d. partition functions over subsets of
the rectangle.

The choice of
$(s,t)$ implies that
$ %\begin{equation}\label{chdir}
p^{(\theta)}(s,t) = p(s,t).
$
Fix $0< \e< (p^{(\theta)}(s,t) - r)/4 $. Fix $m \in\bN$ large
enough so that $m(s\wedge t) \ge1$ and
%
%e4.9 #&#
\begin{equation}\label{grainsize}
\bE\log Z_{\lfloor{ms}\rfloor, \lfloor{mt}\rfloor} > m(r + 2\e).
\end{equation}
%
%This can be done since \eqref{chdir} is in effect.

Let $B^{k,\ell}_{a,b} = \{ a, \ldots, a+ k-1 \}\times\{b, \ldots,
\ell+b -1\}$ denote
the $k\times\ell$ rectangle with lower left corner at $(a,b)$. For
$i, \ell\ge0 $ define
pairwise disjoint $\lfloor{ms}\rfloor\times\lfloor{mt}\rfloor$ rectangles
\[
B^i_{\ell} = B^{\lfloor{ms}\rfloor, \lfloor{mt}\rfloor}_{(\ell
+i)\lfloor{ms}\rfloor-\ell+1,
\ell\lfloor{mt}\rfloor+1}.
\]
Define a diagonal union of these rectangles by
$ %\begin{equation}\label{tile-diag}
\Delta_i =
\bigcup_{\ell\ge0} B^i_\ell
% \end{equation}
$, $i\ge0$; see Figure~\ref{lb-fig3}.

%f3 #&#
\begin{figure}

\includegraphics{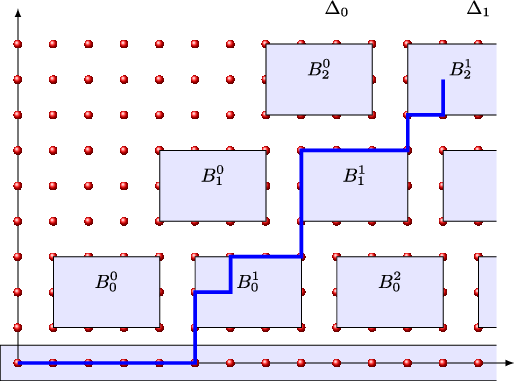}

\caption{The $\lfloor{ms}\rfloor\times\lfloor{mt}\rfloor$
rectangles and the diagonals
$\Delta_i$
in the proof of Lemma \protect\ref{left-lm}. The thickset line is a
lattice path that is counted in $Z_1$.}\label{lb-fig3}
\end{figure}

Let $M = \lfloor{n^a}\rfloor\lfloor{ms}\rfloor$. This is the range
of diagonals $\Delta_i$
we consider.
Then we cut the diagonals off before they exit the $\lfloor{ns}\rfloor
\times\lfloor{nt}\rfloor$ rectangle.
Let $N=N(n)$ be the maximal integer such that $B^M_N$ lies in
$[0, \lfloor{ns}\rfloor]\times[0, \lfloor{nt}\rfloor]$.
Diagonal $\Delta_M$ exits the $\lfloor{ns}\rfloor\times\lfloor
{nt}\rfloor$ rectangle
through the east edge,
and consequently
there exist positive constants $c_m$, $C_m$ such that
%
%e4.10 #&#
\begin{eqnarray}\label{Nbounds}
\lfloor{ns}\rfloor- c_m < N\lfloor{ms}\rfloor+ \lfloor {ms}\rfloor
\bigl\lfloor{n^a}\bigr\rfloor&\le&\lfloor{ns}\rfloor\quad
\mbox{and}\nonumber\\[-8pt]\\[-8pt]
\lfloor{nt}\rfloor- C_m n^a < N\lfloor {mt}\rfloor&\le&
\lfloor{nt}\rfloor.\nonumber
\end{eqnarray}

Having defined the cutoff $N$, define the remaining diagonals by
$\Delta^n_i =
\bigcup_{0\le\ell\le N} B^i_\ell$ for $0\le i\le M$. These
diagonals lie
in $[0, \lfloor{ns}\rfloor]\times[0, \lfloor{nt}\rfloor]$.
Fix a path $\pi$ that proceeds horizontally from point $(N\lfloor
{ms}\rfloor,
N\lfloor{mt}\rfloor+1)$ to
$(\lfloor{ns}\rfloor, N\lfloor{mt}\rfloor+1)$ and then vertically
up to $(\lfloor{ns}\rfloor, \lfloor{nt}\rfloor)$.
The number of lattice points on $\pi$ is a constant multiple of $n^a$.

For $0\le i\le M$, let $Z_i$ denote the partition function of paths
$x_\centerdot$
of the following type: $x_\centerdot$ proceeds along the $x$-axis from
the origin
to $(i\lfloor{ms}\rfloor+1, 0)$, enters $\Delta^n_i$ at $(i\lfloor
{ms}\rfloor+1, 1)$, and
stays in $\Delta^n_i$ until
it exits from the upper right corner of $B^i_N$ with a vertical step
that connects it with
$\pi$. After that $x_\centerdot$
follows $\pi$ to $(\lfloor{ns}\rfloor, \lfloor{nt}\rfloor)$. The
number $K$ of points
on $x_\centerdot$ outside $\Delta^n_i$ is independent of $i$ and
bounded by a constant multiple of $ n^a$. Let
\[
X = \min\bigl\{Y_x\dvtx x\in\pi\mbox{ or } x \in\bigl\{
(i,0)\dvtx 0\le i \le M\bigr\} \bigr\} %\end{equation}
\]
be the minimal weight outside $\Delta^n_i$ encountered by any path
$x_\centerdot$
of $Z_i$, for any
$0\le i\le M$.

Let $Z^{\Delta}_i$ be the partition function of all lattice paths in
$\Delta^n_i$ from
the lower left corner of $B^i_0$ to the upper right corner of $B^i_N$. Then
$ %\begin{equation}\label{Zibound}
Z_i \ge X^K Z^{\Delta}_i,
$
and consequently
%
%e4.11 #&#
\begin{eqnarray}\label{boundstep1}
\quad\bP\bigl\{ \log Z^{(\theta)}_{\lfloor{ns}\rfloor, \lfloor{nt}\rfloor} \le nr \bigr\}
&\le& \bP \Biggl\{
\log\sum_{i=0}^{M} X^K
Z^{\Delta}_i \le nr \Biggr\}
\nonumber
\\[-2pt]
& = & \bP \Biggl\{ K \log X + \log\sum_{i=0}^{M}
Z^{\Delta}_i \le nr \Biggr\}
\nonumber\\[-9pt]\\[-9pt]
& \le & \bP\{ K \log X \le- n\e\} \nonumber\\[-2pt]
&&{}+ \bP \Biggl\{ \log\sum
_{i=0}^{M} Z^{\Delta}_i \le n(r+
\e) \Biggr\}.\nonumber
\end{eqnarray}
Explicit computation with the gamma distribution and $K\le cn^a$ give
the probability $\bP\{ K \log X \le- n\e\} \le e^{-n^2}$ for large $n$.

The $\{Z^{\Delta}_i\}$ are i.i.d., and $Z^{\Delta}_0$ is a product of
the i.i.d. partition
functions $Z^0_k$ of the individual rectangles $B^0_k$ whose mean was controlled
by (\ref{grainsize}). A~standard large deviation estimate for an
i.i.d. sum
gives
\begin{eqnarray*}
\bP \Biggl\{ \log\sum_{i=0}^{M}
Z^{\Delta}_i \le n(r+\e) \Biggr\} &\le& \bP\bigl\{ \log
Z^{\Delta}_0 \le n(r+\e)\bigr\}^M
\\[-4pt]
&=&\bP \Biggl\{ \sum_{k=0}^{N} \log
Z^0_k \le n(r+\e) \Biggr\}^M
\\[-2pt]
&=&\bP \Biggl\{ \sum_{k=0}^{n/m + o(n)} \log
Z^0_k \le n(r+\e) \Biggr\}^M \\[-2pt]
&\le&
e^{-cnM} \le e^{-c_1n^{1+a}}.
\end{eqnarray*}
Putting these bounds back on line (\ref{boundstep1}) completes the
proof of
(\ref{left-stat}).
\end{pf}\eject

The main work resides in proving the following right tail
result.
%
%pr4.3 #&#
\begin{proposition}\label{ldp-pr1} Let $(s,t)\in\bR^2$. Then for all $r\in\bR$,
$\rf_{s,t}(r)$ is given by
%
%e4.12 #&#
\begin{equation}\label{phiexact}
\rf_{s,t}(r) %= \sup_{(\xi, v)\in\mathcal{R}}\big\{ (r,-s)\cdot(
= \sup_{\xi\in[0, \mu)} \Bigl\{ r
\xi-\inf_{\theta\in(\xi, \mu
)} \bigl( t\lmgf_{\theta}(\xi)-s\lmgf_{\mu-\theta}(-
\xi) \bigr) \Bigr\}.
\end{equation}
\end{proposition}

Before turning to the proof of Proposition \ref{ldp-pr1} let us
observe how Theorem \ref{FULLldp} follows.
\begin{pf*}{Proof of Theorem \ref{FULLldp}}
Only a few simple observations are required.
Start by defining $I_{s,t}$ as given in (\ref{fdp}). Then formula
(\ref{Jgam})
that connects $I_{s,t}$ and $\rf_{s,t}$ is established by
(\ref{phiexact}) and by knowing that
$\rf_{s,t}(r)=0$ for $r\le\pmu(s,t)$ (Theorem~\ref{J-thm1}). The
regularity properties of
$I_{s,t}$ follow from the general properties of $\rf$ in Theorems
\ref{J-thm1} and \ref{J-thm2}.

The upper large deviation bound (\ref{fu1}) is built into (\ref
{left-5}) and (\ref{J2}).

For the lower large deviation bound (\ref{fu2}), we consider three
cases:

\begin{longlist}[(iii)]
\item[(i)]
If $\pmu(s,t)\in G$, then $\bP\{ n^{-1}\log Z_{\lfloor
{ns}\rfloor, \lfloor{nt}\rfloor} \in G\}\to1$ and
(\ref{fu2}) holds trivially because its right-hand side is $\le0$.

\item[(ii)] If $G\subseteq(-\infty, \pmu(s,t))$,
(\ref{fu2}) holds trivially because its right-hand side is $-\infty$.

\item[(iii)] The remaining case is the one where $G$ contains an interval
$(a,b)\subset(\pmu(s,t), \infty)$. Since the distribution is
continuous including
$a$ into $G$ makes no difference, and so
\begin{eqnarray*}
&&
n^{-1}\log\bP\bigl\{ n^{-1}\log Z_{\lfloor{ns}\rfloor, \lfloor
{nt}\rfloor} \in G
\bigr\}
\\
&&\qquad \ge n^{-1}\log \bigl( \bP\{ \log Z_{\lfloor{ns}\rfloor,
\lfloor{nt}\rfloor} \ge na\} - \bP\{
\log Z_{\lfloor{ns}\rfloor, \lfloor{nt}\rfloor} \ge nb\} \bigr)
\\
&&\qquad \longrightarrow-\rf_{s,t}(a),
\end{eqnarray*}
where the limit follows from (\ref{J2}) and the strict increasingness
of $\rf_{s,t}$
on $[\pmu(s,t), \infty)$
which implies that for large enough $n$,
\[
\bP\{ \log Z_{\lfloor{ns}\rfloor, \lfloor{nt}\rfloor} \ge nb\} \le e^{-n \rf_{s,t}(a)
-n\e}
\]
for some $\e>0$. We can take $a=\inf G\cap(\pmu(s,t), \infty)$ and then
$\rf_{s,t}(a) =\inf_{r\in G\cap(\pmu(s,t), \infty
)}I_{s,t}(r) = \inf_{r\in G} I_{s,t}(r)$.\qed
\end{longlist}
\noqed\end{pf*}

The remainder of the section is devoted to proving Proposition \ref{ldp-pr1}.
%Let us first lay out the strategy before we get lost in analytic
%details.
Again we begin with the decomposition (\ref{deco}) of the stationary
partition function.
Inside the sums on the right-hand side of (\ref{deco}) we have
partition functions with i.i.d. Gamma$^{-1}(\mu)$-weights
$\{Y_{i,j}\}$
whose large deviations we wish to extract.
But we do not know the large deviations of $\log Z^{(\theta)}_{\lfloor
{ns}\rfloor, \lfloor{nt}\rfloor}$,\vadjust{\goodbreak} so at\vspace*{1pt}
first the decomposition seems unhelpful.
To get around the problem, use
definition (\ref{fractionid}) to write
\[
\log Z^{(\theta)}_{\lfloor{ns}\rfloor, \lfloor{nt}\rfloor} - \log Z^{(\theta)}_{0, \lfloor{nt}\rfloor} =
\sum_{j=1}^{\lfloor{ns}\rfloor} \log U_{i,\lfloor{nt}\rfloor}.
\]
By Proposition \ref{burke-pr} we have a sum of i.i.d.'s on the right,
whose large deviations
we can immediately write down by Cram\'er's theorem. To take advantage
of this,
divide through (\ref{deco}) by $Z^{(\theta)}_{0, \lfloor{nt}\rfloor
}=\prod_{j=1}^{\lfloor{nt}\rfloor} V_{0,j} $
to rewrite it as
%
%e4.13 #&#
\begin{eqnarray}
\label{deco1} \prod_{i=1}^{\lfloor{ns}\rfloor}U_{i,\lfloor{nt}\rfloor}
&=&\sum_{\l=1}^{\lfloor{nt}\rfloor} \Biggl( \prod
_{j=\ell+1}^{\lfloor{nt}\rfloor} V_{0,j}^{-1} \Biggr)
Z^{\square
}_{(1,\l),(\lfloor{ns}\rfloor,\lfloor{nt}\rfloor)}
\nonumber\\[-8pt]\\[-8pt]
&&{} +\sum_{k=1}^{\lfloor{ns}\rfloor} \Biggl( \prod
_{j=1}^{\lfloor
{nt}\rfloor} V_{0,j}^{-1} \Biggr)
\Biggl( \prod_{i=1}^{k}U_{i,0}
\Biggr) Z^{\square}_{(k,1),(\lfloor{ns}\rfloor,\lfloor{nt}\rfloor
)}.
\nonumber
\end{eqnarray}
To compactify notation we use a convention where the $y$-axis is
labeled by negative
indices and
introduce these quantities:
%
%e4.14 #&#
\begin{equation}\label{sums}
\mbox{for $k\in\bZ$}\qquad \eta_k = \cases{ \displaystyle \prod
_{j=-k+1}^{\lfloor{nt}\rfloor} V_{0,j}^{-1}, &\quad $k
\le0$,
\vspace*{2pt}\cr
\displaystyle \Biggl( \prod_{j=1}^{\lfloor{nt}\rfloor}
V_{0,j}^{-1} \Biggr) \prod_{i=1}^{k}U_{i,0},
&\quad $k \ge 1$,}
\end{equation}
where an empty product equals 1 by definition, and
%
%e4.15 #&#
\begin{equation}\label{vk}
\mbox{for $z\in\bR$}\qquad \vvec(z)=\cases{ \bigl(1,\lfloor{-z}\rfloor\bigr),
&\quad $z
\le-1$,
\vspace*{1pt}\cr
(1,1), &\quad $-1<z<1$,
\vspace*{1pt}\cr
\bigl(\lfloor{z}\rfloor,1\bigr), &\quad $z\ge1$.}
\end{equation}
Then (\ref{deco1}) rewrites as
%
%e4.16 #&#
\begin{equation}
\label{L-Zdec} \prod_{i=1}^{\lfloor{ns}\rfloor}U_{i,\lfloor{nt}\rfloor}
=\mathop{\sum_{ k = -\lfloor{nt}\rfloor
}}_{ k \neq0}^{\lfloor{ns}\rfloor}
\eta_k Z^{\square}_{\vvec(k),(\lfloor
{ns}\rfloor,\lfloor{nt}\rfloor)}
\end{equation}
from which we extract these inequalities:
%
%e4.17 #&#
\begin{eqnarray}
\label{UB}
&&
\log\eta_k + \log Z^{\square}_{\vvec(k),(\lfloor{ns}\rfloor
,\lfloor{nt}\rfloor)} \nonumber\\
&&\qquad\le
\sum_{i=1}^{\lfloor{ns}\rfloor}\log U_{i,\lfloor
{nt}\rfloor}
\\
&&\qquad \le \mathop{\max_{-\lfloor{nt}\rfloor\le k
\le\lfloor{ns}\rfloor
}}_{ k\ne0} \bigl\{\log
\eta_k + \log Z^{\square}_{\vvec
(k),(\lfloor{ns}\rfloor,\lfloor{nt}\rfloor)} \bigr\}
+ \log\bigl(
n(s+t)\bigr).\nonumber
\end{eqnarray}
These inequalities will be the basis for proving Proposition \ref{ldp-pr1}.

We record the right tail rate functions for the random variables in
(\ref{UB}).

For the i.i.d. weights
$\{U_{i, \lfloor{nt}\rfloor}\}$ we have the right branch of the Cram\'
er rate
function
%
%e4.18 #&#
\begin{eqnarray}\label{JJ}
R_s(r) &=& -\lim_{n \rightarrow\infty} n^{-1}\log\bP \Biggl\{
\sum_{i=1}^{\lfloor{ns}\rfloor}\log U_{i,\lfloor{nt}\rfloor}\ge nr
\Biggr\}\nonumber\\[-8pt]\\[-8pt]
&=& \cases{ s I_{\theta}\bigl(rs^{-1}\bigr), &\quad $r\ge-s
\Psi_0(\theta)$,
\vspace*{1pt}\cr
0, &\quad $r< -s\Psi_0(\theta)$.}\nonumber
\end{eqnarray}
The rate function $I_{\theta}$ defined by (\ref{JJ}) is given by
%
%e4.19 #&#
\begin{equation}\label{logcram!!!}\qquad
I_{\theta}(r) %&= -\lim_{n\to\infty} n^{-1} \log\bP\Bigg\{ \sum_{i=1}^{n}Y_i \ge r
=-r\Psi^{-1}_0(-r)-
\log\Gamma\bigl(\Psi^{-1}_0(-r)\bigr)+\theta r+ \log
\Gamma(\theta),\qquad r\in\bR.
\end{equation}

The convex dual of $R_s$ is given by
%
%e4.20 #&#
\begin{equation}\label{J*}
R^*_s(\xi) = \cases{ s\log\Gamma(\theta- \xi)-s\log\Gamma(
\theta),&\quad $0\le\xi< \theta$,
\cr
\infty, &\quad $\xi<0$ or $\xi\ge\theta$,}
\end{equation}
and we emphasize that it can be finite only when $\theta> \xi\ge0$.

For real $a \in[-t,s]$,
%
%e4.21 #&#
\begin{equation}\label{liadef}
\kappa_a(r)= -\lim_{n\rightarrow\infty}n^{-1}\log\bP\{\log
\eta_{\lfloor{na}\rfloor} \ge nr \}
\end{equation}
exists and is finite, convex and continuous in $r$. (For $a\le0$ it is
simply a Cram\'er rate
function for an i.i.d. sum, and for $a>0$ we can use Lemma~\ref{sumsind}.)
The convex dual is
%
%e4.22 #&#
\begin{eqnarray}
\label{l*} \kappa^*_a(\xi) &=& \sup_{r\in\bR}\bigl\{ \xi r -
\kappa_a(r)\bigr\}
\nonumber\\[-8pt]\\[-8pt]
&=& \cases{ (t+a) \bigl(\log\Gamma(\mu-\theta+ \xi) - \log\Gamma(\mu- \theta )
\bigr), \vspace*{2pt}\cr
\qquad\hspace*{22pt}-t\le a\le0, \xi\geq 0,
\vspace*{2pt}\cr
t \bigl(\log\Gamma(\mu-\theta+ \xi) - \log
\Gamma(\mu- \theta ) \bigr)\vspace*{2pt}\cr
\qquad{}+ a \bigl(\log\Gamma(\theta- \xi) - \log \Gamma(
\theta) \bigr),
\vspace*{2pt}\cr
\qquad\qquad\hspace*{0pt} 0< a \le s, 0 \le\xi< \theta,
\vspace*{2pt}\cr
\infty, \hspace*{6pt}\qquad\mbox{otherwise}.}\hspace*{-20pt}\nonumber
\end{eqnarray}
The derivation of (\ref{l*}) is similar to that of (\ref{J*}) from
(\ref{JJ}).
Note that there is a discontinuity in $\kappa_a$ and $\kappa^*_a$ as
$a$ passes through $0$.
The rightmost zero $m_{\kappa,a}$ of $\kappa_a$ is the law of large
numbers limit,
%
%e4.23 #&#
\begin{equation}\label{mka}\qquad
m_{\kappa,a}=\lim_{n\to\infty} \frac{\log\eta_{\lfloor
{na}\rfloor}}n =\cases{ (t+a)
\Psi_0(\mu-\theta), &\quad $-t\le a\le0$,
\cr
t\Psi_0(\mu-
\theta)-a\Psi_0(\theta), &\quad $0<a\le s$.}
\end{equation}
In contrast to the functions $\kappa_a$ and $\kappa_a^*$, $m_{\kappa
,a}$ is continuous at $a=0$.
Introduce the ``macroscopic'' version of (\ref{vk}): for real $a$,
%
%e4.24 #&#
\begin{equation}\label{vbara}
n^{-1}\vvec(na)\to\vvbar(a)=\cases{ (0,-a),&\quad $-t \le a \le0$,
\cr
% (0,0), & a= 0\\
(a, 0), &\quad $0\le a \le s$.}
\end{equation}
With this notation we have, again for real $a \in[-t,s]$,
for the partition functions that appear in (\ref{L-Zdec}), the
following large deviations:
%
%e4.25 #&#
\begin{equation}\label{J6}
\rf_{(s,t) - \vvbar(a)}(r)= - \lim_{n\to\infty}n^{-1}\log\bP\bigl\{
\log Z^{\square}_{\vvec(na), (\lfloor{ns}\rfloor,\lfloor
{nt}\rfloor)} \ge nr\bigr\}.
\end{equation}
We used Lemma \ref{lm-J10} to take care of the small discrepancy between
$(\lfloor{ns}\rfloor,\lfloor{nt}\rfloor)-\vvec(na)$ and
$\lfloor{n((s,t) - \vvbar(a))}\rfloor$,
unless $a=-t$ or $a=s$ when
this is a case of i.i.d. large deviations, and therefore simpler.

Let $m_{\kappa,a}$ and $m_{\rf, b}$ be the
rightmost zeroes of $\kappa_a$ and $\rf_{(s,t) - \vvbar(b)}$, respectively.
For $(a, b) \in[-t,s]^2$, let
%
%e4.26 #&#
\begin{eqnarray}
\label{Hdef} H^{a,b}_{s,t}(r) &=& \lim_{n\to\infty}n^{-1}
\log\bP\bigl\{ \log\eta_{\lfloor{na}\rfloor} + \log Z^{\square}_{\vvec(nb), (\lfloor
{ns}\rfloor,\lfloor{nt}\rfloor)} \ge
nr\bigr\}
\nonumber\\[-8pt]\\[-8pt]
& =&  \cases{ 0,\qquad r < m_{\kappa,a}+ m_{\rf, b},
\vspace*{2pt}\cr
\displaystyle
\inf_{m_{\kappa,a}\le x \le r- m_{\rf, b}}\bigl\{ \kappa_a(x) + \rf_{(s,t) -\vvbar(b)}(r-x)
\bigr\}, \vspace*{2pt}\cr
\hspace*{33.1pt}r \ge m_{\kappa,a}+ m_{\rf, b}.}\nonumber
\end{eqnarray}
The existence of $H^{a,b}_{s,t}(r)$ and the second equality follow
from Lemma \ref{sumsind}. We need some regularity:
%
%le4.4 #&#
\begin{lemma}\label{Hcont-lm} Fix $0<s,t<\infty$ and a compact set $K\subseteq\bR$. Then
$H^{a,b}_{s,t}(r)$ is uniformly continuous as a function of $(b,r)\in
[-t,s]\times K$, uniformly
in $a\in[-t,s]$. That is,
%
%e4.27 #&#
\begin{equation}\label{Hcont}
\lim_{\delta\searrow0} \mathop{\sup_{a,b,b'\in
[-t,s], r,x \in K: }}_{
\vert b-b'\vert\le\delta, \vert r-x\vert\le\delta} \bigl\vert
H^{a,b}_{s,t}(r) - H^{a,b'}_{s,t}(x)\bigr\vert=0.
\end{equation}
\end{lemma}
\begin{pf} This follows from the explicit formula in (\ref{Hdef}).
First, we have the joint
continuity $(b,r)\mapsto\rf_{(s,t) -\vvbar(b)}(r)$ from Theorem \ref
{J-thm2}.
Second, we argue that $x$ in the infimum can be restricted to a single
compact set
simultaneously for $(a,b,r)\in[-t,s]^2\times K$. That $m_{\kappa,a}$
is bounded
is evident from (\ref{mka}). To show that the upper bound $r- m_{\rf
,b}$ of $x$
is bounded above, we need to show a lower bound on $m_{\rf,b}=\pmu
((s,t)-\vvec(b))$.
A lower bound on the free energy is easy: by discarding all but a
single path,
\[
\pmu\bigl((s,t)-\vvec(b)\bigr) = \lim_{n\to\infty}n^{-1} \log
Z^{\square
}_{\lfloor{n((s,t) - \vvbar(b))}\rfloor} \ge-\bigl(s+t-\vert b\vert\bigr)\Psi_0(\mu).
\]
\upqed
\end{pf}

We abbreviate $H^a_{s,t}(r)= H^{a,a}_{s,t}(r)$.

The unknown rate functions $\rf_{s,t}$ are now inside (\ref{Hdef}), while
the other rates $R_s$ and $\kappa_a$ we know explicitly. The next
lemma is the
counterpart of (\ref{UB}) in terms of rate functions.
%
%le4.5 #&#
\begin{lemma}\label{decomp}
Let $s,t > 0$ and $r\in\bR$. Then
%
%e4.28 #&#
\begin{equation}\label{Jdefi}
R_s(r) =\inf_{-t\le a \le s }H^a_{s,t}(r).
%= \inf_{-t\le a \le s}
\end{equation}
\end{lemma}
\begin{pf}
For any $a \in[-t,s]$, by the first inequality of (\ref{UB}),
%
%e4.29 #&#
\begin{eqnarray}\label{bo}
-R_s(r)&=&\lim_{n \rightarrow\infty}n^{-1}\log\bP \Biggl\{ \sum
_{i=1}^{\lfloor{ns}\rfloor} \log U_{i,\lfloor{nt}\rfloor} \ge nr
\Biggr\}
\nonumber
\\
&\ge&\lim_{n \rightarrow\infty} n^{-1} \log\bP \bigl\{ \log
\eta_{\lfloor{na}\rfloor} + \log Z^{\square}_{\vvec(na),(\lfloor
{ns}\rfloor,\lfloor{nt}\rfloor)} \ge nr \bigr\}
\\
&\ge& - H^{a}_{s,t}(r).
\nonumber
\end{eqnarray}
Supremum over $a\in[-t,s]$ on the right gives $\le$ in (\ref{Jdefi}).

To get $\ge$ in (\ref{Jdefi}) we use the second inequality of (\ref
{UB}) together
with a partitioning argument. Let $\e> 0$.
Note this technical point about handling the errors of the
partitioning. With $B, \delta>0$,
Chebyshev's inequality and
the l.m.g.f. of (\ref{lmgfom}) give the bound
%
%e4.30 #&#
\begin{equation}\label{Ybd1}\qquad
\bP \Biggl\{ \sum_{i=1}^{\lfloor{n\delta}\rfloor} \log
Y_{i,1} \le -n\e \Biggr\} \le e^{ -nB
(\e- B^{-1}\delta\log({\Gamma(\mu+B)}/{\Gamma
(\mu)})  )} \le e^{ -B\e n/2},
\end{equation}
where the second inequality comes from choosing $\delta=\delta(\e,
B)$ small enough.
The right tail for $\log Y$ does not give such a bound with an
arbitrarily large $B$.
Consequently we arrange the errors so that they can be bounded as above.

Given $B>0$, fix a small enough $\delta> 0$ and let $-t = a_0 < a_1 <
\cdots< a_q=0 < \cdots< a_m = s$ be a partition of the interval
$-[t,s]$ so that $|a_{i+1}-a_i| < \delta$.
%We estimate
%&\le e^{-\xi n\e}\big(\bE e^{\xi\log U^{-1}}\big)^{n\delta}\nonumber\\
%&\le e^{-n(\xi\e- \delta\int_{\mu-\theta}^{\mu-\theta+ \xi}\Psi_0(x)
% dx)}\nonumber\\
%&\le e^{-n\xi\e/2} \label{tiny}.
We illustrate how a term with index $k$ from the right-hand side of
(\ref{UB}) is reduced to
a term involving only partition points.
Consider the
case $a_{i} \ge0$ and let $\lfloor{na_i}\rfloor\le k\le\lfloor
{na_{i+1}}\rfloor$:
%
%e4.31 #&#
\begin{eqnarray}\label{tiny}
&&
\bP \bigl\{ \log\eta_{k} + \log Z^{\square}_{\vvec(k),(\lfloor
{ns}\rfloor,\lfloor{nt}\rfloor)} \ge
nr \bigr\}
\nonumber
\\
&&\qquad\le\bP \Biggl\{ \log\eta_{\lfloor{na_{i+1}}\rfloor} + \log Z^{\square
}_{\vvec(na_i),(\lfloor{ns}\rfloor,\lfloor{nt}\rfloor)}
\nonumber
\\
&&\hspace*{12pt}\qquad\quad{} -\sum_{j=k+1}^{\lfloor{na_{i+1}}\rfloor
}\log U_{j,0} -
\sum_{j=\lfloor{na_i}\rfloor}^{k-1}\log Y_{j,1} \ge nr
\Biggr\}
\nonumber\\
&&\qquad\le\bP \bigl\{ \log\eta_{\lfloor{na_{i+1}}\rfloor} + \log Z^{\square}_{\vvec
(na_i),(\lfloor{ns}\rfloor,\lfloor{nt}\rfloor)}
\ge n(r-\e) \bigr\}
\\
&&\qquad\quad{} + \bP \Biggl\{ -\sum_{j=k+1}^{\lfloor
{na_{i+1}}\rfloor}\log
U_{j,0} -\sum_{j=\lfloor{na_i}\rfloor}^{k-1} \log
Y_{j,1} \ge n\e \Biggr\}
\nonumber
\\
&&\qquad\le\bP \bigl\{ \log\eta_{\lfloor{na_{i+1}}\rfloor} + \log Z^{\square}_{\vvec
(na_i),(\lfloor{ns}\rfloor,\lfloor{nt}\rfloor)}
\ge n(r-\e) \bigr\}\nonumber\\
&&\qquad\quad{} + e^{-B \e n/2} .\nonumber
\end{eqnarray}

On the other hand, if
$a_i<0$ and $\lfloor{-na_{i+1}}\rfloor< -k\le\lfloor{-na_i}\rfloor
$, then we would develop
as follows:
\begin{eqnarray*}
&&\log\eta_k + \log Z^{\square}_{\vvec(k),(\lfloor{ns}\rfloor
,\lfloor{nt}\rfloor)}
\\
&&\qquad\le\log\eta_{\lfloor{na_i}\rfloor} -\sum_{j=-k+1}^{-\lfloor
{na_i}\rfloor}
\log V_{0,j} + \log Z^{\square}_{\vvec(na_{i+1}),(\lfloor{ns}\rfloor,\lfloor
{nt}\rfloor)} \\
&&\qquad\quad{}- \sum
_{j=\lfloor{-na_{i+1}}\rfloor\vee1}^{-k-1} \log Y_{1,j}
\end{eqnarray*}
and get the same bound as on line (\ref{tiny}) but with $a_{i}$ and
$a_{i+1}$ switched
around.

Now for $\ge$ in (\ref{Jdefi}). Assume $n$ is large enough so that
$n\e> \log(ns+nt)$. Starting
from (\ref{UB}),
\begin{eqnarray*}
&&
n^{-1}\log\bP \Biggl\{ \sum_{i=1}^{\lfloor{ns}\rfloor}
\log U_{i,\lfloor{nt}\rfloor} \ge nr \Biggr\}
\\
&&\qquad\le\mathop{\max_{-\lfloor{nt}\rfloor\le k \le\lfloor
{ns}\rfloor}}_{ k\ne0} n^{-1} \log\bP \bigl
\{ \log\eta_{k} + \log Z^{\square}_{\vvec(k),(\lfloor
{ns}\rfloor,\lfloor{nt}\rfloor)} \ge n(r - \e)
\bigr\}
\\
&&\qquad\quad{} +n^{-1}\log(ns+nt)
\\
&&\qquad\le \max_{0\le i \le q-1}n^{-1} \log \bigl(\bP \bigl\{ \log
\eta_{\lfloor
{na_{i}}\rfloor} + \log Z^{\square}_{\vvec(na_{i+1}),(\lfloor{ns}\rfloor,\lfloor
{nt}\rfloor)} \ge n(r-2\e ) \bigr\}
\\
&&\hspace*{248.5pt}\qquad\quad{} + e^{-B \e n/2} \bigr)
\\
&&\qquad\quad{} \vee\max_{q\le i \le m-1}n^{-1} \log \bigl(\bP \bigl\{
\log \eta_{\lfloor{na_{i+1}}\rfloor} + \log Z^{\square}_{\vvec
(na_i),(\lfloor{ns}\rfloor,\lfloor{nt}\rfloor)}\\
&&\qquad\quad\hspace*{196pt} \ge n(r-2\e)
\bigr\}
\\
&&\hspace*{207.2pt}\qquad\quad{} + e^{-B \e n/2} \bigr) + \e.
\end{eqnarray*}
Take $n \to\infty$ above to obtain
\begin{eqnarray*}
-R_s(r)&\le& \Bigl\{ \max_{0\le i \le q-1} \bigl(-H^{a_{i}, a_{i+1}}_{s,t}(r-2
\e ) \bigr) \vee (-B\e/2 ) \Bigr\}
\\
&&{} \vee \Bigl\{ \max_{q\le i \le m-1} \bigl(-H^{a_{i+1},
a_i}_{s,t}(r-2
\e) \bigr) \vee (-B\e/2 ) \Bigr\} + \e
\\
&\le&\sup_{a,b\in[-t,s]\dvtx \vert a-b\vert\le\delta} \bigl(-H^{a,b}_{s,t}(r-2\e) \bigr)
\vee (-B\e/2 ) + \e.
\end{eqnarray*}
We first let $\delta\searrow0$, and by Lemma \ref{Hcont-lm} the
bound above becomes
\[
-R_s(r)\le\sup_{a\in[-t,s]} \bigl(-H^{a,a}_{s,t}(r-2
\e) \bigr) \vee (-B\e/2 ) + \e.
\]
Next we take $B\nearrow\infty$, and finally $\e\searrow0$ with
another application of
Lemma \ref{Hcont-lm}. This establishes $\ge$ in (\ref{Jdefi}).
\end{pf}

A key analytic trick will be to look at the dual $ \rf_{(t,t)-\vvbar
(a)}^*(\xi)$
of the right tail rate
as a function of $a$. This lemma will be helpful.
%
%le4.6 #&#
\begin{lemma}\label{Gisconvex}
For a fixed $\xi\in[0, \mu)$, the function
%
%e4.32 #&#
\begin{equation}
G_{\xi}(a) = \cases{ - \rf_{(t,t)-\vvbar(a)}^*(\xi),&\quad $a\in[0,t]$,
\cr
\infty, &\quad$a<0$ or $a>t$,}
\end{equation}
is continuous on $[0,t]$, and convex and lower semi-continuous on $\bR
$. In particular, $G^{**}_{\xi}(a) = G_{\xi}(a)$
for $a\in\bR$.
\end{lemma}
\begin{pf} To show convexity on $[0,t]$, let $\lambda\in(0,1)$ and
$a = \lambda a_1+(1-\lambda)a_2$:
%
%e4.33 #&#
\begin{eqnarray}\label{Gconv}\quad
&&
- \rf_{(t,t)-\vvbar(a)}^*(\xi) \nonumber\\
&&\qquad= -\sup_{r\in\bR}\bigl\{\xi r -
\rf_{(t,t)-\vvbar(a)}(r)\bigr\}
\nonumber
\\
&&\qquad=\inf_{r\in\bR}\bigl\{ \rf_{t-a,t}(r)-\xi r\bigr\}
\nonumber
\\
&&\qquad\le\inf_{r\in\bR} \mathop{\inf_{(r_1,r_2): }}_{ \lambda r_1 +
(1-\lambda)r_2=r} \bigl\{
\lambda\bigl( \rf_{t-a_1,t}(r_1)-\xi r_1\bigr)+(1-
\lambda) \bigl( \rf_{t-a_2,t}(r_2)-\xi r_2\bigr)
\bigr\}
\\
&&\qquad= \inf_{(r_1,r_2)\in\bR^2}\bigl\{\lambda\bigl( \rf_{t-a_1,t}(r_1)-
\xi r_1\bigr)+(1-\lambda) \bigl( \rf_{t-a_2,t}(r_2)-
\xi r_2\bigr)\bigr\}
\nonumber
\\
&&\qquad=\lambda\inf_{r_1 \in\bR}\bigl\{ \rf_{t-a_1,t}(r_1)-\xi
r_1\bigr\} +(1-\lambda) \inf_{r_2\in\bR}\bigl\{
\rf_{t-a_2,t}(r_2)-\xi r_2\bigr\}
\nonumber
\\
&&\qquad=-\lambda\rf_{t-a_1,t}^*(\xi)-(1-\lambda)
\rf_{t-a_2,t}^*(\xi).\nonumber
\end{eqnarray}
The inequality comes from the convexity of $ \rf$ in the variable $(t-a,t,r)$.

For finiteness on $[0,t]$ it is now enough to show that $G_{\xi}(a)$
is finite at the endpoints.
Continuity then follows in the interior $(0,t)$.
First take $a=t$. Then $\rf_{0,t}^*$ is the dual of a Cram{\'e}r rate
function, and for $\xi\ge0$
%
%e4.34 #&#
\begin{equation}\label{va}
G_{\xi}(t) = - \rf_{0,t}^*(\xi) = - t\log\bE
e^{ \xi\log Y_{1,0} },
\end{equation}
which is finite for $\xi<\mu$.\eject

Convexity of $ \rf_{s,t}(r)$ and symmetry $ \rf_{s,t}(r)= \rf_{t,s}(r)$ imply $ \rf_{t,t}(r) \le\rf_{0,2t}(r)$.

From this
%
%e4.35 #&#
\begin{eqnarray}
G_{\xi}(0) &=& - \rf_{t,t}^*(\xi) = \inf_{r\in\bR}\bigl
\{ \rf_{t,t}(r) - \xi r\bigr\}
\nonumber\\[-8pt]\\[-8pt]
&\le&\inf_{r\in\bR}\bigl\{ \rf_{0,2t}(r)-\xi r \bigr\} = -
\rf_{0,2t}^*(\xi ) < \infty.\nonumber
\end{eqnarray}

\textit{Continuity at $a=0$}.
To show that $G_{\xi}$ is also continuous at the endpoints, we first
obtain a lower bound. For any $r \in\bR$,
\[
\rf^{*}_{t-a,t}(\xi) \ge r\xi- \rf_{t-a,t}(r)
\]
hence, by continuity of $\rf_{s,t}$ in the $(s,t)$ argument,
%
%e4.36 #&#
\begin{equation}
\varliminf_{a\to0}\rf^{*}_{t-a,t}(\xi) \ge r\xi-
\rf_{t,t}(r).
\end{equation}
Supremum over $r$ gives $\varliminf_{a\to0}\rf^{*}_{t-a,t}(\xi) \ge
\rf^{*}_{t,t}(\xi)$.

For the upper bound, let $0<a<t$.
Varadhan's theorem (Theorem 4.3.1 in \cite{demb-zeit}) applies in the
present setting.
This is justified in the proof of Corollary \ref{lmgf-thm1} below and
another similar
justification is given for (\ref{supcond}) below.
Consequently,
%
%e4.37 #&#
\begin{eqnarray}
\rf^{*}_{t,t}(\xi)&=& \lim_{n\to\infty}n^{-1}
\log\bE e^{\xi\log
Z_{\lfloor{nt}\rfloor,\lfloor{nt}\rfloor}}
\nonumber
\\
&\ge&\lim_{n\to\infty}n^{-1}\log\bE e^{\xi\log Z_{\lfloor
{n(t-a)}\rfloor,\lfloor{nt}\rfloor}}
\nonumber\\[-8pt]\\[-8pt]
&&{} + \lim_{n\to\infty}n^{-1}\log\bE e^{\xi\sum_{i=\lfloor{n(t-a)}\rfloor+1}^{\lfloor{nt}\rfloor}\log
Y_{i, \lfloor{nt}\rfloor}
}
\nonumber
\\
&=&\rf_{t-a,t}^*(\xi)+ a \log\bE Y^{\xi}.\nonumber
\end{eqnarray}
Taking $a \searrow0$ yields continuity at $a=0$.\vspace*{8pt}

\textit{Continuity at $a=t$}.
The lower bound follows as in the previous case. For the upper bound we
use a path counting argument. Let $e^{n F(s,t)}$ be an upper bound on
the number of paths in
$\Pi_{\lfloor{ns}\rfloor, \lfloor{nt}\rfloor} $ such that $F(0+,t)=0$.
Consider first the case where $0\le\xi<1$.
Then
%
%e4.38 #&#
\begin{eqnarray}
\label{aux411} \rf_{t-a,t}^*(\xi) &=& \lim_{n\to\infty}n^{-1}
\log\bE \Biggl(\sum_{x_\centerdot\in\Pi_{(\lfloor{n(t-a)}\rfloor, \lfloor{nt}\rfloor
)} }\prod
_{i=1}^{\lfloor{nt}\rfloor+\lfloor{n(t-a)}\rfloor} Y_{x_i} \Biggr)^\xi
\nonumber
\\
&\le&\lim_{n\to\infty}n^{-1}\log\sum_{x_\centerdot\in\Pi
_{(\lfloor{n(t-a)}\rfloor, \lfloor{nt}\rfloor)} }
\prod_{i=1}^{\lfloor{nt}\rfloor+\lfloor{n(t-a)}\rfloor}\bE (Y)^\xi
\\
&=& F(t-a,t) + (2-a/t)\rf_{0,t}^*(\xi).
\nonumber
\end{eqnarray}
For $1\le\xi<\mu$, Jensen's inequality yields
%
%e4.39 #&#
\begin{equation}\label{aux412}
\rf_{t-a,t}^*(\xi)\le\xi F(t-a,t) +(2-a/t)\rf_{0,t}^*(\xi).
\end{equation}
Let $a\nearrow t$ to get the continuity.

$G^{**}_{\xi} = G_{\xi}$ is a consequence of convexity and lower
semicontinuity, by
\cite{rock-ca}, Theorem 12.2.
\end{pf}
\begin{pf*}{Proof of Proposition \ref{ldp-pr1}}
The remainder of the proof is
convex analysis. The goal is to derive the following formula for
the right tail rate function $\rf_{s,t}$:
%
%e4.40 #&#
\begin{equation}\label{J17}
\rf_{s,t}(r) %= \sup_{(\xi, v)\in\mathcal{R}}\big\{ (r,-s)\cdot(
= \sup_{\xi\in[0, \mu)} \Bigl\{ r
\xi-\inf_{\theta\in(\xi, \mu
)} \bigl( t\lmgf_{\theta}(\xi)-s\lmgf_{\mu-\theta}(-
\xi) \bigr) \Bigr\}.
\end{equation}

We begin by expressing the
explicitly known dual $R^{*}_s(\xi)$ from (\ref{J*}) in terms of the
unknown function
$\rf_{(s,t)-\vvbar(a) }$.
%Fix $\xi\in[0, \mu)$.
Equation (\ref{Hdef}) says that $H^a_{s,t}$ is the infimal convolution
of $\kappa_a$
and $\rf_{(s,t)-\vvbar(a)}$, in symbols
$
H^a_{s,t}= \kappa_{a}\square\rf_{(s,t)-\vvbar(a) }.
$
By Theorem~16.4 in \cite{rock-ca} addition is dual to infimal convolution.
Starting with (\ref{Jdefi}) we have
%
%e4.41 #&#
\begin{eqnarray}
\label{aux69} R^{*}_s(\xi) %& = \sup_{r}\big\{ r\xi+ \sup_{-t\le a\le s}\big\{ - (
& = &
\sup_{-t\le a\le s}\sup_{r\in\bR} \bigl\{ r\xi- ( \kappa_{a}
\square\rf_{(s,t)-\vvbar(a) }) (r) \bigr\}
\nonumber
\\
& = &\sup_{-t\le a\le s} ( \kappa_{a}\square\rf_{(s,t)-\vvbar(a)
})^*(
\xi)
\\
& = &\sup_{-t\le a\le s} \bigl\{ \kappa_{a}^*(\xi)+
\rf_{(s,t)-\vvbar(a)}^*(\xi) \bigr\}.
\nonumber
\end{eqnarray}
Combining this with (\ref{J*}) gives, for $ 0 \le\xi< \theta$,
%
%e4.42 #&#
\begin{equation}\label{aux7}
s\log\Gamma(\theta- \xi)-s\log\Gamma(\theta) = \sup_{-t\le a\le
s} \bigl\{
\kappa_{a}^*(\xi)+ \rf_{(s,t)-\vvbar(a) }^*(\xi) \bigr\}.
\end{equation}

Now regard $\xi\in[0, \mu)$ fixed, and let $\theta\in(\xi, \mu)$
vary. Introduce
temporary definitions
%
%e4.43 #&#
\begin{equation}\label{wow}
u_a(\theta) = \cases{ -\hks{(\theta)}=\lmgf_{\mu-\theta}(-\xi) =
\log\Gamma(\mu -\theta+ \xi) - \log\Gamma(\mu- \theta),
\vspace*{2pt}\cr
\hspace*{213pt}-t\le a\le0,
\vspace*{2pt}\cr
d_{\xi}{(\theta)}=\lmgf_{\theta}(\xi) = \log\Gamma(\theta- \xi )
- \log\Gamma(\theta), \qquad 0< a\le s.}\hspace*{-28pt}
\end{equation}
Substitute (\ref{l*}) and (\ref{wow}) into equation (\ref{aux7}) to get
%
%e4.44 #&#
\begin{equation}\label{almost}
s\log\frac{\Gamma(\theta- \xi)}{ \Gamma(\theta)} -t\log\frac
{\Gamma(\mu-\theta+ \xi) }{ \Gamma(\mu- \theta)}= \sup_{-t\le
a\le s} \bigl\{
au_a(\theta)+ \rf_{(s,t)-\vvbar(a) }^*(\xi) \bigr\}.\hspace*{-34pt}
\end{equation}
The right-hand side begins to resemble a convex dual, and will allow us
to solve for $\rf_{s,t}$.
We can specialize to the case $s=t$ because $(t,t)-\vvbar(a)$ gives
all the pairs
$(s,t)$ with $0\le s\le t$.
When $s=t$, the $\rf_{s,t}=\rf_{t,s}$ symmetry allows us to write
(\ref{almost}) as
\[
t\bigl(\dks(\theta)+\hks(\theta)\bigr)= \sup_{0\le a\le t} \bigl\{
a \bigl(\hks(\theta)\vee\dks(\theta) \bigr)+ \rf_{t-a,t}^*(\xi) \bigr\},
\]
and it splits into cases as follows:
\[
t\bigl(\dks(\theta)+\hks(\theta)\bigr)=
\cases{ \displaystyle \sup_{0\le a\le t} \bigl\{ a
\hks(\theta)+ \rf_{t-a,t}^*(\xi) \bigr\}, &\quad$\theta\in\bigl[(\mu+\xi)/2,
\mu\bigr)$,
\vspace*{2pt}\cr
\displaystyle \sup_{0\le a\le t} \bigl\{ a \dks(\theta)+ \rf_{t-a,t}^*(\xi)
\bigr\}, &\quad$\theta\in\bigl(\xi, (\mu+\xi)/2\bigr]$.}
\]
We can discard one of the branches above. For if $\theta'=\mu+\xi
-\theta$, then
$\dks(\theta')=\hks(\theta)$, and we see that the two equations
given by the two branches
are in fact equivalent. So we restrict to the case
$\theta\in[(\mu+\xi)/2,\mu)$ and continue with
%
%e4.45 #&#
\begin{equation}\label{almost2}
t\bigl(\dks(\theta)+\hks(\theta)\bigr)= \sup_{0\le a\le t} \bigl\{ a \hks (
\theta)+ \rf_{t-a,t}^*(\xi) \bigr\}.
\end{equation}
The function $\hks$ is strictly increasing, so we can change variables via
$v=\hks(\theta)$ between the intervals $\theta\in[(\mu+\xi)/2,\mu
)$ and
$v\in[\hks((\mu+ \xi)/2),\infty)$.
Recall also $G_{\xi}(a) = - \rf_{t-a,t}^*(\xi)$ from Lemma \ref
{Gisconvex}. This turns (\ref{almost2}) into
%
%e4.46 #&#
\begin{eqnarray}\label{almost3}
t\bigl(\bigl(d_{\xi}\circ\hks^{-1}\bigr) (v)+ v\bigr)&=&
\sup_{0\le a\le t} \bigl\{ av- G_{\xi
}(a) \bigr\} \nonumber\\[-8pt]\\[-8pt]
&=&
G^*_{\xi}(v),\qquad \hks\biggl(\frac{\mu+ \xi}2\biggr)\le v <
\infty.\nonumber
\end{eqnarray}

Utilizing $G_\xi=G_\xi^{**}$, we get the following expression for the
rate function~$\rf$:
%
%e4.47 #&#
%e4.48 #&#
\begin{eqnarray}\label{aux10}\qquad
\rf_{t-a,t}(r) &=& \sup_{\xi\in[0,\mu)}\bigl\{r\xi- \rf_{t-a,t}^*(
\xi )\bigr\} % \mbox{by \eqref{dcd} and Theorem \ref{phidef}},
= \sup_{\xi\in[0,\mu)}\bigl\{r\xi+ G_{\xi}(a)
\bigr\}
\\
&=& \sup_{\xi\in[0,\mu)} \Bigl\{r\xi+ \sup_{v\in\bR} \bigl[ av-
G^*_{\xi}(v) \bigr] \Bigr\}
\nonumber
\\
% \\ &=\sup_{\xi\in[0,\mu)} \sup_{v\in\bR} \{r\xi+ av - G^*_{\xi}(v)\}
&=& \sup_{\xi\in[0,\mu)} \Bigl\{r\xi+
\sup_{v\in[\hks((\mu+ \xi
)/2),\infty)} \bigl[ av- G^*_{\xi}(v) \bigr] \Bigr\}
\nonumber
\\
\label{aux11}
&=& \sup_{\xi\in[0,\mu)} \Bigl\{r\xi+ \sup_{v\in[\hks((\mu+ \xi
)/2),\infty)} \bigl[ (a-t)v- t\dks
\bigl(\hks^{-1}(v)\bigr) \bigr] \Bigr\}.
\end{eqnarray}
In the next to last equality above, we restricted the supremum over $v$
to the interval
$v\in[\hks((\mu+ \xi)/2),\infty)$. This is justified because
$G_\xi^*$ is convex,
$a\ge0$ and from (\ref{almost3}) we can compute the right derivative
$(G^{*}_{\xi})' (\hks(\frac{\mu+ \xi}2)+ )=0$. The
restriction of the
supremum then allows us to replace $G^*_{\xi}(v)$ with (\ref{almost3}).

The proof is complete. In the case $0<s\le t$, take $a=t-s$ on line
(\ref{aux10}).
Line (\ref{aux11})
is the desired representation for $\rf_{s,t}$. It turns into (\ref
{J17}) by the
$v$ to $\theta$ change of
variable. The case $s>t$ follows from the symmetry
$ \rf_{s,t}(r)= \rf_{t,s}(r)$.
\end{pf*}

The next lemma makes explicit the formula(s) for $\rf_{s,t}^*$ that
were implicit
in the proof of Proposition \ref{ldp-pr1}.
%
%le4.7 #&#
\begin{lemma}\label{critical}
Let $s, t\ge0$ and $\xi\in[0,\mu)$. Then
%
%e4.49 #&#
%e4.50 #&#
\begin{eqnarray}\label{cd1}
\rf_{s,t}^*(\xi) &=& \inf_{\rho\in(\xi,\mu)} \bigl\{ t
\lmgf_{\rho
}(\xi)-s\lmgf_{\mu-\rho}(-\xi) \bigr\}
\\
\label{cd2}
&=& \inf_{\theta\in(\xi,\mu)} \bigl\{ s\lmgf_{\theta}(\xi)-t
\lmgf_{\mu-\theta}(-\xi) \bigr\}.
\end{eqnarray}
\end{lemma}
\begin{pf}
(\ref{cd2}) comes from (\ref{cd1})
by the change of variable $\rho= \mu+ \xi-\theta$.
Comparison of the two shows that we can assume $s\le t$. To
prove (\ref{cd1}) for $s\le t$, start from Lemma \ref{Gisconvex}:
\[
\rf^*_{s,t}(\xi) = - G_{\xi}(t-s) = - G_{\xi}^{**}(t-s)
= - \sup_{v\in\bR} \bigl\{ (t-s)v- G^*_{\xi}(v)\bigr\}.
\]
Restrict the supremum as in (\ref{aux10}) and (\ref{aux11}), substitute
in (\ref{almost3}) and change variables from $v$ to $\theta=\hks^{-1}(v)$.
\end{pf}
\begin{pf*}{Proof of Corollary \ref{lmgf-thm1}}
If $\xi\ge\mu$,
\[
\xi\log Z_{\lfloor{ns}\rfloor, \lfloor{nt}\rfloor} \ge\sum_j \xi \log
Y_{x_j}
\]
for any particlar path
$x_\centerdot\in\Pi_{\lfloor{ns}\rfloor,\lfloor{nt}\rfloor} $,
and then $\Lambda_{s,t}(\xi)=\infty$ comes from $M_\mu(\xi)=\infty
$ from (\ref{lmgfom}).

Let $\xi<\mu$. Pick $\gamma>1$ such that $\gamma\xi<\mu$. Then the
bound
\[
\sup_n n^{-1} \log\bE e^{ \gamma\xi\log Z_{\lfloor{ns}\rfloor,
\lfloor{nt}\rfloor}} < \infty
\]
follows from path counting, as in (\ref{aux411}) and (\ref{aux412}).
This bound is sufficient for Varadhan's theorem
(Theorem 4.3.1 in \cite{demb-zeit}) which gives
\begin{eqnarray*}
\lim_{n\to\infty} \Lambda_{s,t}(\xi)&=&n^{-1} \log\bE
e^{ \xi\log Z_{\lfloor
{ns}\rfloor, \lfloor{nt}\rfloor}} = I^*_{s,t}(\xi) = \sup_{r\in\bR} \bigl\{ r\xi-
I_{s,t}(r) \bigr\}
\\
& = & \sup_{r\ge\pmu(s,t)} \bigl\{ r\xi- I_{s,t}(r) \bigr\}=
\sup_{r\ge\pmu
(s,t)} \bigl\{ r\xi- \rf_{s,t}(r) \bigr\}.
\end{eqnarray*}
We discarded $\{I_{s,t}=\infty\}=\{r< \pmu(s,t)\}$ from the supremum.
Since $I_{s,t}$ increases for $r\ge\pmu(s,t)$, the case $\xi\le0$ of
(\ref{finpieces}) follows. For $\xi\ge0$ the values $\rf_{s,t}(r)=0$ for
$r< \pmu(s,t)$ can be put back in because they do not alter the supremum.
Consequently
$\Lambda_{s,t}(\xi)=\rf_{s,t}^*(\xi)$ for $\xi\ge0$. Lemma \ref
{critical} completes
the proof of this corollary.
\end{pf*}

There is nothing new in the proof of Corollary \ref{free-rate},
so we omit it.
% \begin{pf}[Proof of Corollary \ref{free-rate}] \end{pf}

%s5 #&#
\section{Proofs for the stationary log-gamma model}
\label{secpf2}

In this section we prove the results of Section \ref{secstat-lg}.
\begin{pf*}{Proof of Theorem \ref{hor-thm}}
Coarse-graining arguments and simple error bounds
readily give the following limit:
\begin{eqnarray*}
p^{{(\theta), \mathrm{hor}}}(s,t)&=&\lim_{n\to\infty} n^{-1} \log
Z^{(\theta), \mathrm{hor}}_{\lfloor{ns}\rfloor, \lfloor{nt}\rfloor
}
\\
&=& \lim_{n\to\infty} \max_{1\le k\le\lfloor{ns}\rfloor} \Biggl( n^{-1} \sum
_{i=1}^{k}\log U_{i,0} +
n^{-1}\log Z^{\square}_{(k,1),(\lfloor{ns}\rfloor,\lfloor
{nt}\rfloor)} \Biggr)
\\
&=& \sup_{0\le a\le s}\bigl\{ -a\Psi_0(\theta) +\pmu(s-a,t) \bigr
\}
\\
&=& \sup_{0\le a\le s} \inf_{0 < \rho< \mu}\bigl\{ -a\Psi_0(
\theta) +(a-s)\Psi_0(\rho) - t\Psi_0(\mu-\rho)\bigr\}.
\end{eqnarray*}
In the last step we substituted in (\ref{pmu}).
Formula (\ref{pres-hor1}) follows from this by some calculus.

From the definition (\ref{Zhor}) of $Z^{(\theta), \mathrm
{hor}}_{\lfloor{ns}\rfloor,\lfloor{nt}\rfloor}$, follow\vspace*{2pt}
inequalities analogous
to (\ref{UB}), and then with arguments like those in the proof of
Lemma \ref{decomp},
we derive a right tail LDP
%
%e5.1 #&#
\begin{eqnarray}\label{hor-ldp1}
&&
\lim_{n\to\infty} n^{-1}\log\bP\bigl\{ Z^{(\theta), \mathrm
{hor}}_{\lfloor{ns}\rfloor,\lfloor{nt}\rfloor}
\ge nr\bigr\}\nonumber\\[-8pt]\\[-8pt]
&&\qquad = -\rf_{\theta,\mathrm{hor}}(r)=-\inf_{a\in[0,s]} (R_a
\square\rf_{s-a,t}) (r),\nonumber
\end{eqnarray}
where $R_a$ is the rate function from (\ref{JJ}).
For $\xi\ge0$ the l.m.g.f. in (\ref{Lhor})
satisfies $\Lambda^{\mathrm{hor}}_{\theta, (s,t)}(\xi) = \rf_{\theta,\mathrm{hor}}^*(\xi)$.
This would be a consequence of Varadhan's theorem if we had a full LDP,
but now
we have to justify this separately, and we do so in Lem\-ma~\ref
{var-aux-lm} below. 1
Proceeding as in (\ref{aux69}) and using (\ref{cd2}),
\begin{eqnarray*}
\Lambda^{\mathrm{hor}}_{\theta, (s,t)}(\xi) &=& \sup_{a\in[0,s]}
\bigl(R_a^*(\xi) +\rf_{s-a,t}^*(\xi) \bigr)
\\
&=& \sup_{a\in[0,s]} \inf_{\rho\in(\xi,\mu)} \bigl\{ a\lmgf_\theta(
\xi)+ (s-a)\lmgf_{\rho}(\xi)-t\lmgf_{\mu-\rho}(-\xi) \bigr\}.
\end{eqnarray*}
Formula (\ref{pres-hor2}) follows from some calculus. The sup and inf
can be interchanged
by a minimax theorem (see, e.g., \cite{kass}), and this makes the
calculus easier.
\end{pf*}
%
%le5.1 #&#
\begin{lemma}\label{var-aux-lm}
Let $Z^{(\theta),\mathrm{hor}}_{\lfloor{ns}\rfloor,\lfloor
{nt}\rfloor}$ the partition\vspace*{1pt}
function given by (\ref{Zhor}), and let $\rf_{\theta,\mathrm{
hor}}(r)$ as given by (\ref{hor-ldp1}). Then for $0\le\xi< \theta$,
\[
\lim_{n\to\infty} n^{-1}\log\bE e^{\xi\log Z^{(\theta),\mathrm
{hor}}_{\lfloor{ns}\rfloor,\lfloor{nt}\rfloor}} =
\sup_{r\in
\bR}\bigl\{ r\xi- \rf_{\theta,\mathrm{hor}}(r) \bigr\} =
\rf^*_{\theta,\mathrm{hor}}(\xi).
\]
\end{lemma}
\begin{pf} Let $0<\xi<\theta$.
Set
\[
\underline\gamma= \varliminf_{n\to\infty} n^{-1}\log\bE
e^{\xi
\log Z^{(\theta),\mathrm{hor}}_{\lfloor{ns}\rfloor,\lfloor
{nt}\rfloor} } \quad\mbox{and}\quad \overline\gamma= \varlimsup_{n\to\infty}
n^{-1}\log\bE e^{\xi
\log Z^{(\theta),\mathrm{hor}}_{\lfloor{ns}\rfloor,\lfloor
{nt}\rfloor} }.
\]
First we have an exponential Chebyshev argument for a lower bound:
\[
n^{-1}\log\bP\bigl\{ \log Z^{(\theta), \mathrm{hor}}_{\lfloor
{ns}\rfloor,\lfloor{nt}\rfloor} \ge nr
\bigr\} \le-\xi r + n^{-1}\log\bE e^{\xi\log Z^{(\theta), \mathrm
{hor}}_{\lfloor{ns}\rfloor,\lfloor{nt}\rfloor} }.
\]
Letting $n\to\infty$ along a suitable subsequence gives
$
\underline{\gamma} \ge\xi r - \rf_{\theta,\mathrm{hor}}(r)
$
for all $r\in\bR$. Thus $\underline{\gamma}\ge\rf_{\theta
,\mathrm{hor}}^*(\xi)$ holds.

For the upper bound we claim that
%
%e5.2 #&#
\begin{equation}\label{T-F}
\lim_{r \to\infty}\varlimsup_{n\to\infty}n^{-1}\log\bE
\bigl(e^{\xi\log Z^{(\theta), \mathrm{hor}}_{\lfloor{ns}\rfloor,\lfloor
{nt}\rfloor} }\mathbf {1}\bigl\{ \log Z^{(\theta), \mathrm{hor}}_{\lfloor{ns}\rfloor,\lfloor
{nt}\rfloor}
\ge nr\bigr\} \bigr)= - \infty.
\end{equation}

Assume for a moment that (\ref{T-F}) holds. To establish the upper
bound let $\delta>0$ and partition $\bR$ with $r_{i}= i\delta$,
$i\in\bZ$:
%
%e5.3 #&#
\begin{eqnarray}\label{v1}
&&
n^{-1}\log\bE \bigl( e^{\xi\log Z^{(\theta), \mathrm
{hor}}_{\lfloor{ns}\rfloor,\lfloor{nt}\rfloor} } \bigr) %& = n^{-1}\log\Big(\sum_{i=-\infty}^{\infty} \bE e^{\xi\log Z^{(
%Z^{(\theta), \mathrm{hor}}_{\fl{ns},\fl{nt}} < nr_{i+1}\}\Big)
\nonumber\\
&&\qquad\le n^{-1}\log \Biggl[ \sum
_{i=-m}^{m} e^{n \xi r_{i+1}}\bP\bigl\{ \log
Z^{(\theta), \mathrm{hor}}_{\lfloor{ns}\rfloor,\lfloor{nt}\rfloor
} \ge nr_i\bigr\}
\\
&&\qquad\quad\hspace*{37pt}{} + e^{n\xi r_{-m}} + \bE \bigl(e^{\xi\log
Z^{(\theta), \mathrm{hor}}_{\lfloor{ns}\rfloor,\lfloor{nt}\rfloor
} }\mathbf{1}\bigl\{ \log
Z^{(\theta), \mathrm{hor}}_{\lfloor{ns}\rfloor,\lfloor{nt}\rfloor
} \ge nr_m\bigr\} \bigr) \Biggr].\nonumber
\end{eqnarray}
By (\ref{T-F}), for each $M>0$ there exists $m=m(M)$ so that
\[
n^{-1}\log\bE \bigl(e^{\xi\log Z^{(\theta), \mathrm{hor}}_{\lfloor
{ns}\rfloor,\lfloor{nt}\rfloor} }\mathbf{1}\bigl\{ \log
Z^{(\theta),
\mathrm{hor}}_{\lfloor{ns}\rfloor,\lfloor{nt}\rfloor} \ge nr_{m}\bigr\} \bigr) < - M.
\]
%
%We can also assume that $m_0(M )> m_0(L)$ when $M>L$.
A limit along a suitable subsequence in (\ref{v1}) yields
\begin{eqnarray*}
\overline{\gamma} &\le& \max_{-m\le i \le m}\bigl\{ \xi r_{i+1} -
\rf_{\theta,\mathrm{hor}}(r_i)\bigr\}\vee\xi r_{-m}\vee(-M)
\\
&\le& \Bigl(\sup_{r\in\bR}\bigl\{ \xi r - \rf_{\theta,\mathrm{hor}}(r)\bigr\} +
\xi\delta \Bigr)\vee\xi r_{-m}\vee(-M).
\end{eqnarray*}
The proof of the lemma follows by letting $\delta\to0$, $m\to\infty
$ and $M\to\infty$.

Now to show (\ref{T-F}). Note that there exists $\alpha>1$ such that
$\alpha\xi< \theta$,
%
%e5.4 #&#
\begin{equation}\label{supcond}
\sup_{n} \bigl( \bE e^{\alpha\xi\log Z^{(\theta), \mathrm
{hor}}_{\lfloor{ns}\rfloor,\lfloor{nt}\rfloor} } \bigr)^{1/n} <
\infty.
\end{equation}
To see this, distinguish cases where $\alpha\xi<1$ or otherwise. Let
$N$ denote the number of paths, and recall that $N \le e^{cn}$ for some $c>0$:
For $\alpha\xi<1$,
\begin{eqnarray*}
\bigl(\bE e^{\alpha\xi\log Z^{(\theta), \mathrm{hor}}_{\lfloor
{ns}\rfloor,\lfloor{nt}\rfloor} } \bigr)^{1/n} &=& \Biggl(\bE \Biggl[ \Biggl(
\sum_{x \in\Pi_{(\lfloor{ns}\rfloor
,\lfloor{nt}\rfloor)}} \prod_{i=1}^{\lfloor{ns}\rfloor+\lfloor{nt}\rfloor}
Y_{x_j} \Biggr)^{\alpha\xi} \Biggr] \Biggr)^{1/n}
\\
&\le& \Biggl(N\prod_{i=1}^{\lfloor{nt}\rfloor+\lfloor{ns}\rfloor}\bE
Y^{\alpha\xi} \Biggr)^{1/n} \le e^{c}
\lmgf_{\theta}(\alpha\xi)^{t+s}.
\end{eqnarray*}
For $\alpha\xi\ge1$, Jensen's inequality gives
\begin{eqnarray*}
\bigl(\bE e^{\alpha\xi\log Z^{(\theta), \mathrm{hor}}_{\lfloor
{ns}\rfloor,\lfloor{nt}\rfloor} } \bigr)^{1/n} &=& \Biggl(\bE \Biggl[ \Biggl(
\sum_{x \in\Pi_{(\lfloor{ns}\rfloor
,\lfloor{nt}\rfloor)}} \prod_{i=1}^{\lfloor{ns}\rfloor+\lfloor{nt}\rfloor}
Y_{x_j} \Biggr)^{\alpha\xi} \Biggr] \Biggr)^{1/n}
\\
&\le& \Biggl(N^{\alpha\xi}\prod_{i=1}^{\lfloor{nt}\rfloor+\lfloor
{ns}\rfloor}
\bE Y^{\alpha
\xi} \Biggr)^{1/n}\e e^{c\alpha\xi}
\lmgf_{\theta}(\alpha\xi )^{t+s}.
\end{eqnarray*}
To show (\ref{T-F}), use H{\"o}lder's inequality,
\begin{eqnarray*}
&&
n^{-1}\log\bE \bigl(e^{\xi\log Z^{(\theta), \mathrm{hor}}_{\lfloor
{ns}\rfloor,\lfloor{nt}\rfloor}}\mathbf{1}\bigl\{ \log
Z^{(\theta),
\mathrm{hor}}_{\lfloor{ns}\rfloor,\lfloor{nt}\rfloor} \ge nr\bigr\} \bigr)
\\
&&\qquad\le\alpha^{-1}\log\sup_{n} \bigl(\bE e^{\alpha\xi\log Z^{(\theta),
\mathrm{hor}}_{\lfloor{ns}\rfloor,\lfloor{nt}\rfloor} }
\bigr)^{1/n}
\\
&&\qquad\quad{} + {(\alpha-1)} {\alpha^{-1}}n^{-1}\log\bP\bigl\{ \log
Z^{(\theta), \mathrm{hor}}_{\lfloor{ns}\rfloor,\lfloor
{nt}\rfloor} \ge nr\bigr\} .
\end{eqnarray*}
Taking a limit $n\to\infty$, we conclude
%
%e5.5 #&#
\begin{equation}
\varlimsup_{n\to\infty}n^{-1}\log\bE \bigl(e^{\xi\log Z^{(\theta
), \mathrm{hor}}_{\lfloor{ns}\rfloor,\lfloor{nt}\rfloor} }\mathbf
{1}\bigl\{ \log Z^{(\theta),
\mathrm{hor}}_{\lfloor{ns}\rfloor,\lfloor{nt}\rfloor} \ge nr\bigr\} \bigr) \le
C_1-C_2\rf_{\theta,\mathrm{hor}}(r)\hspace*{-25pt}
\end{equation}
for positive constants $C_1, C_2$. Letting $r\to\infty$ finishes the proof
because
\[
\lim_{r\to\infty} \rf_{\theta,\mathrm{hor}}(r)=\infty.
\]
%
% since $\rf_{\theta,\mathrm{hor}}(r)$ is a non-constant nondecreasing
%convex function.
\upqed\end{pf}
\begin{pf*}{Proof of Theorem \ref{La-th-thm}} We can assume
$0<\xi<\theta\wedge(\mu-\theta)$ because otherwise the boundary
variables alone
force the l.m.g.f. to blow up.

Let us record the counterpart of
(\ref{pres-hor2}) for $Z^{(\theta), \mathrm{hor}}_{\lfloor
{ns}\rfloor,\lfloor{nt}\rfloor}$. Condition (\ref{trans2}) becomes
%
%e5.6 #&#
\begin{equation}\label{trans3}
t \bigl(\Psi_0(\mu-\theta)-\Psi_0(\mu-\theta-\xi)
\bigr)\ge s \bigl(\Psi_0(\theta+\xi)-\Psi_0(\theta)
\bigr).
\end{equation}
The conclusion becomes that the limit in (\ref{Lver}) exists and is
given by
%
%e5.7 #&#
\begin{equation}\label{pres-ver2}
\Lambda^{\mathrm{ver}}_{\theta, (s,t)}(\xi) = \cases{ t\lmgf_{\mu-\theta}(
\xi)- s\lmgf_{\theta}(-\xi), &\quad if (\ref{trans3}) holds,
\cr
\Lambda_{t,s}(\xi)= \Lambda_{s,t}(\xi), &\quad if (\ref{trans3})
fails.}
\end{equation}

The logarithmic limits lead to the formula
%
%e5.8 #&#
\begin{equation}\label{aux27}
\Lambda_{\theta, (s,t)}(\xi) = \Lambda^{\mathrm
{hor}}_{\theta, (s,t)}(\xi) \vee
\Lambda^{\mathrm{ver}}_{\theta, (s,t)}(\xi),
\end{equation}
and we need to justify that
this is the same as
the maximum in (\ref{La-th}). This comes from several observations:

\begin{enumerate}[(iii)]
\item[(i)] $\Lambda_{s,t}(\xi)=\rf_{s,t}^*(\xi)$ is always bounded above
by the first branches of both
(\ref{pres-hor2}) and (\ref{pres-ver2}). This is evident from
equations (\ref{cd1}) and (\ref{cd2}).

\item[(ii)]
Conditions (\ref{trans2}) and (\ref{trans3}) together define three
ranges for $(s,t)$:
\begin{enumerate}[(a)]
\item[(a)] (\ref{trans2}) and (\ref{trans3}) both hold if and only
if $\alpha_1 t\le s\le\alpha_2 t$;
\item[(b)] (\ref{trans2}) holds and (\ref{trans3}) fails if and only
if $s> \alpha_2 t$;
\item[(c)] (\ref{trans2}) fails and (\ref{trans3}) holds if and only
if $s< \alpha_1 t$.
\end{enumerate}
The constants $0<\alpha_1<\alpha_2$ can be read off (\ref{trans2})
and (\ref{trans3}),
and the strict inequalities are justified by the strict concavity of
$\Psi_0$.

\item[(iii)] In the maximum in (\ref{La-th}), we have
%
%e5.9 #&#
\begin{equation}\label{aux28}
s\lmgf_{\theta}(\xi)- t\lmgf_{\mu-\theta}(-\xi) \ge t
\lmgf_{\mu-\theta}(\xi)- s\lmgf_{\theta}(-\xi)
\end{equation}
if and only if $s\ge\alpha_3 t$ for a constant $\alpha_3>0$ that can
be read off from above.
Strict concavity of $\Psi_0$ implies that $0<\alpha_1<\alpha_3<\alpha_2$.

Now we argue that
%
%e5.10 #&#
\begin{equation}\label{aux29}
\Lambda_{\theta, (s,t)}(\xi) = \max \bigl\{ s\lmgf_{\theta}(\xi)- t
\lmgf_{\mu-\theta}(-\xi), t\lmgf_{\mu-\theta}(\xi)- s\lmgf_{\theta}(-
\xi) \bigr\}.
\end{equation}
This is clear in case (a) as this maximum is exactly $ \Lambda^{\mathrm{hor}}_{\theta, (s,t)}(\xi) \vee
\Lambda^{\mathrm{ver}}_{\theta, (s,t)}(\xi)$. In case (b), $\Lambda^{\mathrm{hor}}_{\theta, (s,t)}(\xi)$
equals\vspace*{1pt} the left-hand side of (\ref{aux28}) which dominates both the
right-hand side
of (\ref{aux28}) and $\Lambda_{s,t}(\xi)$. Consequently in case (b) also
(\ref{aux27}) is the same as (\ref{aux29}).
Case (c) is symmetric to (b). This completes the proof of (\ref{aux29}).

With one additional observation we can verify Remark \ref{La-rem5}.
Namely, $\Lambda_{s,t}(\xi)$ is in fact
\textit{strictly} bounded above by the first branch of either
(\ref{pres-hor2}) or (\ref{pres-ver2}).
The claim is easily verifiable when either of conditions (b) or (c) are
in effect. To see the strict domination when (a) holds, note that
the unique minimizers in formulas (\ref{cd1}) and (\ref{cd2})
are linked by $\rho= \mu+ \xi-\theta$.
But if these formulas matched both first branches in (\ref{pres-hor2})
and (\ref{pres-ver2}),
the connection would have to be $\rho= \mu-\theta$.
This together with (\ref{aux29}) implies that $\Lambda_{s,t}(\xi)<
\Lambda_{\theta, (s,t)}(\xi)$
for all $\theta\in(0,\mu)$.\qed
\end{enumerate}
\noqed\end{pf*}

\section*{Acknowledgments}

We would like to thank an anonymous referee for a thorough reading of
the original manuscript and suggestions that improved the presentation
of this article.

%suskaldyti doi

% imsref loaded by lrinkeviciute, 2012-08-27 11:56:09
% imsref loaded by lrinkeviciute, 2012-08-27 12:21:29

\printaddresses


\begin{thebibliography}{34}
% BibTex style file: ims.bst, 2012-08-21
% Default style options (sort=0,type=number).
% Used options (sort=1,type=number).

%b1 ###
\bibitem{Amir-Cor-Qual}
\begin{barticle}[mr]
\bauthor{\bsnm{Amir},~\bfnm{Gideon}\binits{G.}},
  \bauthor{\bsnm{Corwin},~\bfnm{Ivan}\binits{I.}} \AND
  \bauthor{\bsnm{Quastel},~\bfnm{Jeremy}\binits{J.}}
(\byear{2011}).
\btitle{Probability distribution of the free energy of the continuum directed
  random polymer in {$1+1$} dimensions}.
\bjournal{Comm. Pure Appl. Math.}
\bvolume{64}
\bpages{466--537}.
\bid{doi={10.1002/cpa.20347}, issn={0010-3640}, mr={2796514}}
\bptok{imsref}%
\end{barticle}
\endbibitem

%b2 ###
\bibitem{Ben-Ari}
\begin{barticle}[mr]
\bauthor{\bsnm{Ben-Ari},~\bfnm{Iddo}\binits{I.}}
(\byear{2009}).
\btitle{Large deviations for partition functions of directed polymers in an
  {IID} field}.
\bjournal{Ann. Inst. Henri Poincar\'e Probab. Stat.}
\bvolume{45}
\bpages{770--792}.
\bid{doi={10.1214/08-AIHP185}, issn={0246-0203}, mr={2548503}}
\bptok{imsref}%
\end{barticle}
\endbibitem

%b3 ###
\bibitem{bolt-cmp-89}
\begin{barticle}[mr]
\bauthor{\bsnm{Bolthausen},~\bfnm{Erwin}\binits{E.}}
(\byear{1989}).
\btitle{A note on the diffusion of directed polymers in a random environment}.
\bjournal{Comm. Math. Phys.}
\bvolume{123}
\bpages{529--534}.
\bid{issn={0010-3616}, mr={1006293}}
\bptok{imsref}%
\end{barticle}
\endbibitem

%b4 ###
\bibitem{Carmona-Hu-Gaussian}
\begin{barticle}[mr]
\bauthor{\bsnm{Carmona},~\bfnm{Philippe}\binits{P.}} \AND
  \bauthor{\bsnm{Hu},~\bfnm{Yueyun}\binits{Y.}}
(\byear{2002}).
\btitle{On the partition function of a directed polymer in a {G}aussian random
  environment}.
\bjournal{Probab. Theory Related Fields}
\bvolume{124}
\bpages{431--457}.
\bid{doi={10.1007/s004400200213}, issn={0178-8051}, mr={1939654}}
\bptok{imsref}%
\end{barticle}
\endbibitem

%b5 ###
\bibitem{come-shig-yosh-03}
\begin{barticle}[mr]
\bauthor{\bsnm{Comets},~\bfnm{Francis}\binits{F.}},
  \bauthor{\bsnm{Shiga},~\bfnm{Tokuzo}\binits{T.}} \AND
  \bauthor{\bsnm{Yoshida},~\bfnm{Nobuo}\binits{N.}}
(\byear{2003}).
\btitle{Directed polymers in a random environment: Path localization and strong
  disorder}.
\bjournal{Bernoulli}
\bvolume{9}
\bpages{705--723}.
\bid{doi={10.3150/bj/1066223275}, issn={1350-7265}, mr={1996276}}
\bptok{imsref}%
\end{barticle}
\endbibitem

%b6 ###
\bibitem{come-shig-yosh-04}
\begin{bincollection}[mr]
\bauthor{\bsnm{Comets},~\bfnm{Francis}\binits{F.}},
  \bauthor{\bsnm{Shiga},~\bfnm{Tokuzo}\binits{T.}} \AND
  \bauthor{\bsnm{Yoshida},~\bfnm{Nobuo}\binits{N.}}
(\byear{2004}).
\btitle{Probabilistic analysis of directed polymers in a random environment: A
  review}.
In \bbooktitle{Stochastic Analysis on Large Scale Interacting Systems}.
\bseries{Advanced Studies in Pure Mathematics}
\bvolume{39}
\bpages{115--142}.
\bpublisher{Math. Soc. Japan}, \blocation{Tokyo}.
\bid{mr={2073332}}
\bptok{imsref}%
\end{bincollection}
\endbibitem

%b7 ###
\bibitem{come-varg-06}
\begin{barticle}[mr]
\bauthor{\bsnm{Comets},~\bfnm{Francis}\binits{F.}} \AND
  \bauthor{\bsnm{Vargas},~\bfnm{Vincent}\binits{V.}}
(\byear{2006}).
\btitle{Majorizing multiplicative cascades for directed polymers in random
  media}.
\bjournal{ALEA Lat. Am. J. Probab. Math. Stat.}
\bvolume{2}
\bpages{267--277}.
\bid{issn={1980-0436}, mr={2249671}}
\bptok{imsref}%
\end{barticle}
\endbibitem

%b8 ###
\bibitem{come-yosh-aop-06}
\begin{barticle}[mr]
\bauthor{\bsnm{Comets},~\bfnm{Francis}\binits{F.}} \AND
  \bauthor{\bsnm{Yoshida},~\bfnm{Nobuo}\binits{N.}}
(\byear{2006}).
\btitle{Directed polymers in random environment are diffusive at weak
  disorder}.
\bjournal{Ann. Probab.}
\bvolume{34}
\bpages{1746--1770}.
\bid{doi={10.1214/009117905000000828}, issn={0091-1798}, mr={2271480}}
\bptok{imsref}%
\end{barticle}
\endbibitem

%b9 ###
\bibitem{Comets-Gregorio-arc}
\begin{barticle}[mr]
\bauthor{\bsnm{Comets},~\bfnm{Francis}\binits{F.}} \AND
  \bauthor{\bsnm{Yoshida},~\bfnm{Nobuo}\binits{N.}}
(\byear{2011}).
\btitle{Branching random walks in space--time random environment: Survival
  probability, global and local growth rates}.
\bjournal{J. Theoret. Probab.}
\bvolume{24}
\bpages{657--687}.
\bid{doi={10.1007/s10959-009-0267-x}, issn={0894-9840}, mr={2822477}}
\bptok{imsref}%
\end{barticle}
\endbibitem

%b10 ###
\bibitem{Corwin}
\begin{bmisc}[auto:STB|2012/08/23|07:51:16]
\bauthor{\bsnm{Corwin},~\bfnm{I.}\binits{I.}}
(\byear{2011}).
\bhowpublished{The Kardar--Parisi--Zhang equation and universality class.
  Available at arXiv:\arxivurl{1106.1596}.}
\bptok{imsref}%
\end{bmisc}
\endbibitem

%b11 ###
\bibitem{corw-oconn-sepp-zygo}
\begin{bmisc}[auto:STB|2012/08/23|07:51:16]
\bauthor{\bsnm{Corwin},~\bfnm{I.}\binits{I.}},
  \bauthor{\bsnm{O'Connell},~\bfnm{N.}\binits{N.}},
  \bauthor{\bsnm{Sepp{\"a}l{\"a}inen},~\bfnm{T.}\binits{T.}} \AND
  \bauthor{\bsnm{Zygouras},~\bfnm{N.}\binits{N.}}
(\byear{2011}).
\bhowpublished{Tropical combinatorics and Whittaker functions. Available at
  arXiv:\arxivurl{1110.3489}.}
\bptok{imsref}%
\end{bmisc}
\endbibitem

%b12 ###
\bibitem{demb-zeit}
\begin{bbook}[mr]
\bauthor{\bsnm{Dembo},~\bfnm{Amir}\binits{A.}} \AND
  \bauthor{\bsnm{Zeitouni},~\bfnm{Ofer}\binits{O.}}
(\byear{1998}).
\btitle{Large Deviations Techniques and Applications},
\bedition{2nd} ed.
\bseries{Applications of Mathematics (New York)}
\bvolume{38}.
\bpublisher{Springer}, \blocation{New York}.
\bid{mr={1619036}}
\bptok{imsref}%
\end{bbook}
\endbibitem

%b13 ###
\bibitem{denholl-polymer}
\begin{bbook}[mr]
\bauthor{\bparticle{den} \bsnm{Hollander},~\bfnm{Frank}\binits{F.}}
(\byear{2009}).
\btitle{Random Polymers}.
\bseries{Lecture Notes in Math.}
\bvolume{1974}.
\bpublisher{Springer}, \blocation{Berlin}.
\bnote{Lectures from the 37th Probability Summer School held in Saint-Flour,
  2007}.
\bid{doi={10.1007/978-3-642-00333-2}, mr={2504175}}
\bptok{imsref}%
\end{bbook}
\endbibitem

%b14 ###
\bibitem{deus-zeit-99}
\begin{barticle}[mr]
\bauthor{\bsnm{Deuschel},~\bfnm{Jean-Dominique}\binits{J.-D.}} \AND
  \bauthor{\bsnm{Zeitouni},~\bfnm{Ofer}\binits{O.}}
(\byear{1999}).
\btitle{On increasing subsequences of {I}.{I}.{D}. samples}.
\bjournal{Combin. Probab. Comput.}
\bvolume{8}
\bpages{247--263}.
\bid{doi={10.1017/S0963548399003776}, issn={0963-5483}, mr={1702546}}
\bptok{imsref}%
\end{barticle}
\endbibitem

%b15 ###
\bibitem{NicossThesis}
\begin{bmisc}[auto:STB|2012/08/23|07:51:16]
\bauthor{\bsnm{Georgiou},~\bfnm{N.}\binits{N.}}
(\byear{2011}).
\bhowpublished{Positive and zero temperature polymer models. Ph.D. thesis,
  Univ. Wisconsin--Madison. Available at arXiv:\arxivurl{1210.0600}.}
\bptok{imsref}%
\end{bmisc}
\endbibitem

%b16 ###
\bibitem{huse-henl}
\begin{barticle}[auto:STB|2012/08/23|07:51:16]
\bauthor{\bsnm{Huse},~\bfnm{D.~A.}\binits{D.~A.}} \AND
  \bauthor{\bsnm{Henley},~\bfnm{C.~L.}\binits{C.~L.}}
(\byear{1985}).
\btitle{Pinning and roughening of domain wall in Ising systems due to random
  impurities}.
\bjournal{Phys. Rev. Lett.}
\bvolume{54}
\bpages{2708--2711}.
\bptok{imsref}%
\end{barticle}
\endbibitem

%b17 ###
\bibitem{imbr-spen}
\begin{barticle}[mr]
\bauthor{\bsnm{Imbrie},~\bfnm{J.~Z.}\binits{J.~Z.}} \AND
  \bauthor{\bsnm{Spencer},~\bfnm{T.}\binits{T.}}
(\byear{1988}).
\btitle{Diffusion of directed polymers in a random environment}.
\bjournal{J. Stat. Phys.}
\bvolume{52}
\bpages{609--626}.
\bid{issn={0022-4715}, mr={0968950}}
\bptok{imsref}%
\end{barticle}
\endbibitem

%b18 ###
\bibitem{joha}
\begin{barticle}[mr]
\bauthor{\bsnm{Johansson},~\bfnm{Kurt}\binits{K.}}
(\byear{2000}).
\btitle{Shape fluctuations and random matrices}.
\bjournal{Comm. Math. Phys.}
\bvolume{209}
\bpages{437--476}.
\bid{doi={10.1007/s002200050027}, issn={0010-3616}, mr={1737991}}
\bptok{imsref}%
\end{barticle}
\endbibitem

%b19 ###
\bibitem{kass}
\begin{barticle}[mr]
\bauthor{\bsnm{Kassay},~\bfnm{G.}\binits{G.}}
(\byear{1994}).
\btitle{A simple proof for {K}\"onig's minimax theorem}.
\bjournal{Acta Math. Hungar.}
\bvolume{63}
\bpages{371--374}.
\bid{doi={10.1007/BF01874462}, issn={0236-5294}, mr={1261480}}
\bptok{imsref}%
\end{barticle}
\endbibitem

%b20 ###
\bibitem{kimj-96}
\begin{barticle}[mr]
\bauthor{\bsnm{Kim},~\bfnm{Jeong~Han}\binits{J.~H.}}
(\byear{1996}).
\btitle{On increasing subsequences of random permutations}.
\bjournal{J. Combin. Theory Ser. A}
\bvolume{76}
\bpages{148--155}.
\bid{doi={10.1006/jcta.1996.0095}, issn={0097-3165}, mr={1405997}}
\bptok{imsref}%
\end{barticle}
\endbibitem

%b21 ###
\bibitem{laco-10}
\begin{barticle}[mr]
\bauthor{\bsnm{Lacoin},~\bfnm{Hubert}\binits{H.}}
(\byear{2010}).
\btitle{New bounds for the free energy of directed polymers in dimension
  {$1+1$} and {$1+2$}}.
\bjournal{Comm. Math. Phys.}
\bvolume{294}
\bpages{471--503}.
\bid{doi={10.1007/s00220-009-0957-3}, issn={0010-3616}, mr={2579463}}
\bptok{imsref}%
\end{barticle}
\endbibitem

%b22 ###
\bibitem{Liu-Watbled-2009}
\begin{barticle}[mr]
\bauthor{\bsnm{Liu},~\bfnm{Quansheng}\binits{Q.}} \AND
  \bauthor{\bsnm{Watbled},~\bfnm{Fr{\'e}d{\'e}rique}\binits{F.}}
(\byear{2009}).
\btitle{Exponential inequalities for martingales and asymptotic properties of
  the free energy of directed polymers in a random environment}.
\bjournal{Stochastic Process. Appl.}
\bvolume{119}
\bpages{3101--3132}.
\bid{doi={10.1016/j.spa.2009.05.001}, issn={0304-4149}, mr={2568267}}
\bptok{imsref}%
\end{barticle}
\endbibitem

%b23 ###
\bibitem{loga-shep-77}
\begin{barticle}[mr]
\bauthor{\bsnm{Logan},~\bfnm{B.~F.}\binits{B.~F.}} \AND
  \bauthor{\bsnm{Shepp},~\bfnm{L.~A.}\binits{L.~A.}}
(\byear{1977}).
\btitle{A variational problem for random {Y}oung tableaux}.
\bjournal{Adv. Math.}
\bvolume{26}
\bpages{206--222}.
\bid{doi={10.1016/0001-8708(77)90030-5}, mr={1417317}}
\bptok{imsref}%
\end{barticle}
\endbibitem

%b24 ###
\bibitem{moreno-2010}
\begin{barticle}[mr]
\bauthor{\bsnm{Moreno},~\bfnm{Gregorio}\binits{G.}}
(\byear{2010}).
\btitle{Convergence of the law of the environment seen by the particle for
  directed polymers in random media in the {$L\sp 2$} region}.
\bjournal{J. Theoret. Probab.}
\bvolume{23}
\bpages{466--477}.
\bid{doi={10.1007/s10959-008-0203-5}, issn={0894-9840}, mr={2644870}}
\bptok{imsref}%
\end{barticle}
\endbibitem

%b25 ###
\bibitem{mori-oconn-07}
\begin{barticle}[mr]
\bauthor{\bsnm{Moriarty},~\bfnm{J.}\binits{J.}} \AND
  \bauthor{\bsnm{O'Connell},~\bfnm{N.}\binits{N.}}
(\byear{2007}).
\btitle{On the free energy of a directed polymer in a {B}rownian environment}.
\bjournal{Markov Process. Related Fields}
\bvolume{13}
\bpages{251--266}.
\bid{issn={1024-2953}, mr={2343849}}
\bptok{imsref}%
\end{barticle}
\endbibitem

%b26 ###
\bibitem{oconn-toda}
\begin{barticle}[auto:STB|2012/08/23|07:51:16]
\bauthor{\bsnm{O'Connell},~\bfnm{N.}\binits{N.}}
(\byear{2012}).
\btitle{Directed polymers and the quantum Toda lattice}.
\bjournal{Ann. Probab.}
\bvolume{40}
\bpages{437--458}.
\bptok{imsref}%
\end{barticle}
\endbibitem

%b27 ###
\bibitem{oconn-yor-01}
\begin{barticle}[mr]
\bauthor{\bsnm{O'Connell},~\bfnm{Neil}\binits{N.}} \AND
  \bauthor{\bsnm{Yor},~\bfnm{Marc}\binits{M.}}
(\byear{2001}).
\btitle{Brownian analogues of {B}urke's theorem}.
\bjournal{Stochastic Process. Appl.}
\bvolume{96}
\bpages{285--304}.
\bid{doi={10.1016/S0304-4149(01)00119-3}, issn={0304-4149}, mr={1865759}}
\bptok{imsref}%
\end{barticle}
\endbibitem

%b28 ###
\bibitem{rock-ca}
\begin{bbook}[mr]
\bauthor{\bsnm{Rockafellar},~\bfnm{R.~Tyrrell}\binits{R.~T.}}
(\byear{1970}).
\btitle{Convex Analysis}.
\bseries{Princeton Mathematical Series}
\bvolume{28}.
\bpublisher{Princeton Univ. Press}, \blocation{Princeton, NJ}.
\bid{mr={0274683}}
\bptok{imsref}%
\end{bbook}
\endbibitem

%b29 ###
\bibitem{sepp98ebp}
\begin{barticle}[mr]
\bauthor{\bsnm{Sepp{\"a}l{\"a}inen},~\bfnm{T.}\binits{T.}}
(\byear{1998}).
\btitle{Coupling the totally asymmetric simple exclusion process with a moving
  interface}.
\bjournal{Markov Process. Related Fields}
\bvolume{4}
\bpages{593--628}.
\bnote{I Brazilian School in Probability (Rio de Janeiro, 1997)}.
\bid{issn={1024-2953}, mr={1677061}}
\bptok{imsref}%
\end{barticle}
\endbibitem

%b30 ###
\bibitem{sepp-ptrf-98}
\begin{barticle}[mr]
\bauthor{\bsnm{Sepp{\"a}l{\"a}inen},~\bfnm{Timo}\binits{T.}}
(\byear{1998}).
\btitle{Large deviations for increasing sequences on the plane}.
\bjournal{Probab. Theory Related Fields}
\bvolume{112}
\bpages{221--244}.
\bid{doi={10.1007/s004400050188}, issn={0178-8051}, mr={1653841}}
\bptok{imsref}%
\end{barticle}
\endbibitem

%b31 ###
\bibitem{sepp-poly}
\begin{barticle}[mr]
\bauthor{\bsnm{Sepp{\"a}l{\"a}inen},~\bfnm{Timo}\binits{T.}}
(\byear{2012}).
\btitle{Scaling for a one-dimensional directed polymer with boundary
  conditions}.
\bjournal{Ann. Probab.}
\bvolume{40}
\bpages{19--73}.
\bid{doi={10.1214/10-AOP617}, issn={0091-1798}, mr={2917766}}
\bptok{imsref}%
\end{barticle}
\endbibitem

%b32 ###
\bibitem{sepp-valk-10}
\begin{barticle}[mr]
\bauthor{\bsnm{Sepp{\"a}l{\"a}inen},~\bfnm{Timo}\binits{T.}} \AND
  \bauthor{\bsnm{Valk{\'o}},~\bfnm{Benedek}\binits{B.}}
(\byear{2010}).
\btitle{Bounds for scaling exponents for a {$1+1$} dimensional directed polymer
  in a {B}rownian environment}.
\bjournal{ALEA Lat. Am. J. Probab. Math. Stat.}
\bvolume{7}
\bpages{451--476}.
\bid{issn={1980-0436}, mr={2741194}}
\bptok{imsref}%
\end{barticle}
\endbibitem

%b33 ###
\bibitem{watbled-arc}
\begin{barticle}[mr]
\bauthor{\bsnm{Watbled},~\bfnm{Fr{\'e}d{\'e}rique}\binits{F.}}
(\byear{2012}).
\btitle{Concentration inequalities for disordered models}.
\bjournal{ALEA Lat. Am. J. Probab. Math. Stat.}
\bvolume{9}
\bpages{129--140}.
\bid{issn={1980-0436}, mr={2904479}}
\bptnote{check year}%
\bptok{imsref}%
\end{barticle}
\endbibitem

\end{thebibliography}
\end{document}